\documentclass[12pt,reqno]{amsart}
\usepackage{amsmath, amsfonts, amssymb, amsthm, amscd, amsbsy, mathtools}
\usepackage{fancyhdr}
\usepackage[usenames,dvipsnames,svgnames,x11names,hyperref]{xcolor}
\usepackage{graphicx}
\usepackage{geometry}
\usepackage[]{hyperref}
\usepackage{verbatim}
\usepackage{dsfont}
\usepackage{enumitem}
\usepackage{xparse}
\usepackage{microtype}
\usepackage{fullpage}
\usepackage{setspace}
\usepackage{parskip} 
\usepackage{amsrefs}
\usepackage{comment}

\numberwithin{equation}{section}
\providecommand{\U}[1]{\protect\rule{.1in}{.1in}}
\hypersetup{
colorlinks=true,
breaklinks=true,
urlcolor=NavyBlue,
linkcolor=Fuchsia,
bookmarksopen=false,
filecolor=black,
citecolor=ForestGreen,
linkbordercolor=red
}
\allowdisplaybreaks
\newtheorem{theorem}{Theorem}[section]

\newtheorem{corollary}[theorem]{Corollary}
\newtheorem{lemma}[theorem]{Lemma}
\newtheorem{remark}[]{Remark}
\newtheorem{example}[theorem]{Example}
\newtheorem{examples}[theorem]{Examples}
\newtheorem{foo}[theorem]{Remarks}

\newtheorem{thmx}{Theorem}

\DeclareMathOperator\arctanh{arctanh}
\newcommand{\norm}[1]{{\left\lVert{#1}\right\rVert}}

\newcommand{\dvg}{\text{div}}
\newcommand{\ric}{\text{Ric}}
\newcommand{\id}{\text{Id}}
\makeatletter
\newcommand{\mathleft}{\@fleqntrue\@mathmargin0pt}
\newcommand{\mathcenter}{\@fleqnfalse}
\makeatother

\def\vint{\mathop{\mathchoice%
          {\setbox0\hbox{$\displaystyle\intop$}\kern 0.22\wd0%
           \vcenter{\hrule width 0.6\wd0}\kern -0.82\wd0}%
          {\setbox0\hbox{$\textstyle\intop$}\kern 0.2\wd0%
           \vcenter{\hrule width 0.6\wd0}\kern -0.8\wd0}%
          {\setbox0\hbox{$\scriptstyle\intop$}\kern 0.2\wd0%
           \vcenter{\hrule width 0.6\wd0}\kern -0.8\wd0}%
          {\setbox0\hbox{$\scriptscriptstyle\intop$}\kern 0.2\wd0%
           \vcenter{\hrule width 0.6\wd0}\kern -0.8\wd0}}%
          \mathopen{}\int}

\newcommand{\hn}{\mathbb{H}^n}
\newcommand{\bn}{\mathbb{B}^n}

\newcommand{\rn}{\mathbb{R}^n}
\newcommand{\sn}{\mathbb{S}^n}

\newcommand{\inth}{\int_{\mathbb{H}^n}}
\newcommand{\cs}{C^\infty_c}

\newcommand{\tx}{\tilde{x}}

\DeclareMathOperator{\csch}{csch}

\begin{document}

\title{The Method of Moving Spheres on the Hyperbolic Space and the Classification of Solutions and the prescribed Q-curvature problem }
\thanks{The second and third authors' research was partly supported by a Simons Collaboration grant from the Simons Foundations.}
\author{Jungang Li, Guozhen Lu, Jianxiong Wang}
\address{Jungang Li: Department of Mathematics\\University of Science and Technology of China\\Heifei, Anhui, China.}
\email{jungangli@ustc.edu.cn}
\address{Guozhen Lu: Department of Mathematics\\University of Connecticut\\Storrs, CT 06269, USA.}
\email{guozhen.lu@uconn.edu}
\address{Jianxiong Wang: Department of Mathematics\\Rutgers University\\Piscataway, NJ 08854, USA.}
\email{jiangxiong.wang@rutgers.edu}
\date{}
\begin{abstract}
The classification of solutions to semilinear partial differential equations, as well as the classification of
critical points of the corresponding functionals, has wide applications in the study of partial differential equations and differential geometry. The classical moving plane method and the method of moving sphere on the Euclidean space $\mathbb{R}^n$ provide an effective approach to capture the symmetry of solutions. To the best of our knowledge, the moving sphere method has yet to be developed on the hyperbolic space $\mathbb{H}^n$. In the present paper, we focus on the following equation
\begin{equation*}
  P_k u = f(u)
\end{equation*}
on hyperbolic spaces $\mathbb{H}^n$, where $P_k$ denotes the GJMS operators on $\mathbb{H}^n$ and
$f : \mathbb{R} \to \mathbb{R}$ satisfies certain growth conditions.
We develop a moving sphere approach on $\mathbb{H}^n$ to obtain the symmetry property
as well as the classification of positive solutions to the above equation.
Our methods also rely on the Helgason-Fourier analysis and Hardy-Littlewood-Sobolev inequalities on hyperbolic space together with a Kelvin transform we introduce on the hyperbolic space in this paper. We also present applications to the higher order prescribed $Q$-curvature problem on the hyperbolic space.

\end{abstract}
\setlength\parindent{0pt}
\maketitle

\section{Introduction}
In the celebrated work of Br\'ezis-Nirenberg \cite{BN}, the existence and nonexistence of positive solutions of the
following equation on a bounded domain in Euclidean space  $\rn$ has been systematically studied:
\begin{equation}\label{BN}
  \begin{cases}
    -\Delta u - \lambda u = u^{\frac{n+2}{n-2}} \ &\textup{in } \Omega,  \\
    u > 0 \ &\textup{in } \Omega, \\
    u = 0 \ &\textup{on } \partial \Omega.
  \end{cases}
\end{equation}
Among other results, Br\'ezis-Nirenberg proved: when $n \geq 4$, a solution exists when $0 < \lambda < \Lambda(\Omega)$, where $\Lambda(\Omega)$ is the first Dirichlet eigenvalue; when $n = 3$ and $\Omega$ is a ball, a solution exists if and only if $\frac{\Lambda(\Omega)}{4} < \lambda < \Lambda(\Omega)$. This result relies on carefully examining the following quotient
$$
  S_{n,\lambda} = \inf_{u \in C_0^\infty (\Omega)} \frac{\int_\Omega |\nabla u |^2 - \lambda u^2 }{ \left( \int_{\Omega} |u|^{\frac{2n}{n-2}} \right)^{\frac{n-2}{n}} }
$$
with the Sobolev best constant
$$
  S_n = \inf_{u \in C_0^\infty (\Omega)} \frac{\int_\Omega |\nabla u |^2 }{ \left( \int_{\Omega} |u|^{\frac{2n}{n-2}} \right)^{\frac{n-2}{2}} },
$$
where the latter quantity has been thoroughly studied by Aubin \cite{Aubin} and  Talenti \cite{Talenti1}. Moreover, through a symmetrization approach, Talenti obtained the extremal function of the Sobolev inequality with the unique form (up to translations and dilations):
$$
  u(x) = \frac{1}{\left( a + b|x|^2 \right)^{\frac{n-2}{2}}}.
$$
On the other hand, in the celebrated work of Gidas-Ni-Nirenberg \cite{GNN1}, they considered the following boundary value problem on the ball $B_R(0)\subset{\rn}$.
\begin{equation*}
    \begin{cases}
        -\Delta u = f(u) &\text{ in } B_R(0)\\
         u=0 &\text{ on } \partial B_R(0),
    \end{cases}
\end{equation*}
where $f$ is of class $C^1$. They proved that any positive solution $u$ in $C^2(\overline{B_R(0)})$ is radially symmetric and decreasing.
Their approach is the so-called moving plane method, which was initiated by Alexandrov \cite{Alexandrov} in the 1950s, and further developed by Serrin \cite{serrin}.
Later, Gidas-Ni-Nirenberg \cite{GNN2} studied the equation $-\Delta u = f(u)$ in the entire space $\rn\setminus\{0,\infty\}, n\geq 3$
with two singularities located at the origin and infinity:
\begin{align}
  u(x)\to+\infty \quad &\text{ as } \quad x\to0. \nonumber\\
  |x|^{n-2}u(x)\to+\infty \quad &\text{ as } \quad x\to\infty.
\end{align}
They showed that the solution is radially symmetric about the origin and decreasing.
Subsequently, Caffarelli-Gidas-Spruck \cite{CGS} considered this equation in a punctured ball when $f$ has critical growth.
To be more precise, they require the nonlinearity $f(t)$ to be a locally nondecreasing Lipschitz function with $f(0)=0$, and to satisfy the following condition:
for sufficiently large $t$, the function $t^{-\frac{n+2}{n-2}}f(t)$ is nonincreasing and $f(t)\geq ct^p$ for some $p\geq \frac{n}{n-2}$.
They showed that the solution $u$ is radially symmetric and decreasing.

The moving plane method was later extended to the study of higher order equations on $\mathbb{R}^n$.
The main challenge here is that the moving plane method heavily depends on the maximum principle, which does not always hold for higher order operators. The alternative Hopf type lemma can only be obtained in some special cases; see \cite{BerchioGazzolaWeth1} for example.
To overcome this obstacle, Chen, Li and Ou \cites{CLO,ChenLiOu2} developed a powerful moving plane method for integral equations, obtaining the symmetry of solutions for higher order and even fractional order equations. More precisely, they proved that for the following equation.
$$
  (- \Delta )^{\frac{\alpha}{2}} u = u^{\frac{n+\alpha}{n - \alpha}}
$$
on $\mathbb{R}^n$, every positive regular solution (i.e. locally $L^{\frac{2n}{n-\alpha}}$ solutions) $u$
is radially symmetric and decreasing about some point $x_0$ and therefore assumes the form (up to some dilations):
$$
  u(x) = \frac{1}{\left( a + b |x - x_0|^2 \right)^{\frac{n-\alpha}{2}}}.
$$
Their work completely classifies all the critical points of the functional corresponding to the Hardy-Littlewood-Sobolev
inequalities of order $\alpha$, whose sharp constant was previously obtained by Lieb \cite{Lieb1}.
Inspired by Chen, Li and Ou's work, the first two authors with Yang in \cite{LLY} recently developed a moving plane method for integral equations
on hyperbolic spaces $\mathbb{H}^n$. We hence obtained, among other results, the symmetry of positive solutions of
the following higher order equations on the whole $\mathbb{H}^n$:
\begin{equation}\label{higherorderHn}
  P_k u - \lambda u = u^{\frac{n+2k}{n-2k}},
\end{equation}
where $P_k$ is the GJMS operator. Indeed, our interest in studying the GJMS operators is natural.
It is known that one of the basic tools in conformal geometry to understand the geometry of the underlying space is by studying the transformations on the space.
In the past decades, there have been plenty of studies on the conformally covariant operators as well as their associated equations.
Generally, let $(M^n, g)$ be a smooth Riemannian manifold of dimension $n\geq 2$ with metric $g$. Two metrics $g, \tilde{g}$ are
conformally related   if
\[\tilde{g} = e^{2w} g, \quad w\in C^\infty(M^n).\]
We say that an operator $A_g$ is conformally covariant of bi-degree $(a, b)$ if under the conformal change of metric $\tilde{g}=e^wg$,
$A$ satisfies
\[A_{\tilde{g}}\varphi = e^{-bw}A_g(e^{aw}\varphi), \text{ for all } \varphi\in C^\infty(M),\]
for some $a,b\in\mathbb{R}$.
One of the most well-known examples of conformally covariant operators is the conformal Laplacian
\[L_g:=-\Delta_g+\frac{n-2}{4(n-1)}R_g,\]
where $\Delta_g$ denotes the Laplace-Beltrami operator and $R_g$ is the scalar curvature on $(M^n,g)$. The conformal Laplacian enjoys the conformal transformation property in the following sense:
\[L_{g_u}(\cdot)=u^{-\frac{n+2}{n-2}}L_g(u\cdot),\]
under the metric change
\[g_u=u^{\frac{4}{n-2}}g,\quad u>0.\]
If one prescribes the constant scalar curvature for the metric $g_u$, then $u$ satisfies the following Yamabe equation
\begin{equation}\label{yamabe}
  -\Delta_gu+\frac{n-2}{4(n-1)}R_gu=\frac{n-2}{4(n-1)}cu^{\frac{n+2}{n-2}}.
\end{equation}

There are many operators besides the conformal Laplacian $L_g$ on general Riemannian manifolds of dimension greater than two, which enjoy a conformal covariance property.
For example, the Paneitz operator $P_2$ of fourth order, which is defined as the bi-Laplacian $(-\Delta_g)^2$ plus the lower order curvature terms,
has been studied by many authors.
Interested readers are also referred to the survey article of Chang \cite{Chang} for more background information on the Paneitz operator.
It turns out that both conformal Laplacian and Paneitz operator are two special cases of a general hierarchy of conformally covariant operators,
which are now known as the GJMS operators introduced by Graham-Jenne-Mason-Sparling in \cite{GJMS2} based on the construction of ambient metric by C. Fefferman-Graham \cite{FeffermanGr}.
We refer the reader to works by C. Fefferman-Graham \cite{FeffermanGr1} and Juhl \cite{GJMS3} for more properties of  GJMS operators.

In the present paper, the background manifold is the hyperbolic space $\hn$ with standard metric, in which the GJMS operators can be inductively defined as follows:
\begin{equation*}
    P_k=P_1(P_1+2)\cdots(P_1+k(k-1)), k\in\mathbb{N}.
\end{equation*}
To illustrate our main results, we recall some previous works and make some comments related to the Problem (\ref{higherorderHn}),
which can be viewed as a higher order version of Br\'ezis-Nirenberg problem on $\mathbb{H}^n$.
Indeed, shortly after Br\'ezis-Nirenberg's work \cite{BN}, considerable effort has been made to extend it to polyharmonic
equations on $\mathbb{R}^n$. One of the higher order versions of Problem (\ref{BN}) can be formulated as follows:
\begin{equation}\label{Pucci-Serrin}
  \begin{cases}
    (- \Delta)^k u = \lambda u + |u|^{q - 2} u \ \ \ &\text{on } \Omega \\
    u = D u = \cdots = D^\alpha u = 0 \ \ \ &\text{on } \partial \Omega \\
    |\alpha| \leq k-1,
  \end{cases}
\end{equation}
where $\Omega\subset \mathbb{R}^n$ is a bounded domain, $n > 2k$ and $q = \frac{2n}{n - 2k}$ is the corresponding critical
Sobolev exponent. Gazzola \cite{Gazzola1} proved that when $n \geq 4k$, for for every $\lambda \in (0 , \Lambda_1( (-\Delta)^k , \Omega))$
there exists a solution $u \in W^{k,2}_0(\Omega)$ to the Dirichlet problem (\ref{Pucci-Serrin}), where $\Lambda_1( (-\Delta)^k , \Omega)$
denotes the corresponding first Dirichlet eigenvalue; when $2k + 1 \leq n \leq 4k - 1$, it was also shown that there
exists  $0 < \overline{\Lambda} < \Lambda_1( (-\Delta)^k , \Omega)) $ such that for every $\lambda \in (\overline{\Lambda}, \Lambda_1( (-\Delta)^k , \Omega)))$
the Dirichlet problem (\ref{Pucci-Serrin}) has a solution $u \in W^{k,2}_0(\Omega)$. Interested readers can also see
\cites{Grunau1,Grunau2,PucciSerrin1,PucciSerrin2} for references in this direction. Br\'ezis-Nirenberg type problems
(e.g. (\ref{BN}), (\ref{Pucci-Serrin}) and (\ref{higherorderHn})) deal with semilinear equations whose nonlinearity is equipped with the critical Sobolev exponent.
Due to the lack of compactness, the classical variational method cannot work directly. Such phenomenon appears in many important areas in mathematical physics, partial differential equations, and conformal geometry.

The Poincar\'e ball model of the $n$-dimensional hyperbolic space $\mathbb{H}^n$ is the unit ball $B(0) \subset \mathbb{R}^n$ equipped with the Poincar\'e metric $ds^2 = \frac{4 (dx_1^2 + \cdots + dx_n^2)}{(1 - |x|^2)^2} $.
The hyperbolic space serves as one of the most important models of Riemannian manifolds with constant curvature, whence the Br\'ezis-Nirenberg type problem on $\mathbb{H}^n$ receives much attention. When $k = 1$, the second order GJMS operator is nothing but the classical conformal Laplacian and the Problem (\ref{higherorderHn}) has been extensively studied in \cites{Benguria1,ManciniSandeep1,Stapelkamp1}.
As for higher order cases, consider the conformal transform from the Euclidean unit ball in $\mathbb{R}^n$ to
 $\mathbb{H}^n$ (the Poincar\'e ball model), the $2k$-th order GJMS operator (throughout this paper, we assume $n>2k$),
where $P_1 = - \Delta_{\mathbb{H}} - \frac{n(n-2)}{4}$ is the conformal Laplacian on $\mathbb{H}^n$, is exactly
the pushforward of $(-\Delta)^k$ upon the conformal transform.
To be more precise, we have the following identity (see \cite{Liu}) :
$$
  P_k u = \left( \frac{2}{1 - |x|^2} \right)^{- (\frac{n}{2} + k)} (-\Delta)^k \left( \left( \frac{2}{1 - |x|^2} \right)^{\frac{n}{2} - k} u \right).
$$
Therefore, Problem (\ref{higherorderHn}) is a natural generalization of Problem (\ref{Pucci-Serrin}).
Due to the complexity of the operator, unlike the treatment for the second order problem, one cannot hope to reduce
the Problem (\ref{higherorderHn}) to $\mathbb{R}^n$ and modify the existing argument on $\mathbb{R}^n$.
To overcome the difficulty, in \cite{LLY} Yang and the first two authors developed a new approach, which combines
the Helgason-Fourier analysis theory, the Hardy-Littlewood-Sobolev inequalities, Hardy-Sobolev-Maz'ya inequalities,
and Green's function estimates for the GJMS operators $P_k$ on hyperbolic spaces. The first two authors established the existence and nonexistence results on the bounded domain in $\mathbb{H}^n$ as well as on the entire space $\mathbb{H}^n$. Moreover, using our moving plane method for integral equations on $\mathbb{H}^n$, we proved that every positive solution of the Problem (\ref{higherorderHn}) is radially symmetric and decreasing with respect to the geodesic distance. More recently, we generalized the symmetry result to the equation with a wider family of nonlinearities
\cite{LLW}. More precisely, among other results, we proved the following theorem.
\begin{thmx}
    Let $k \geq 2$, consider the following equation on $\mathbb{H}^n$:
    \begin{equation*}
      P_k u = f(u),
    \end{equation*}
    where $f$ is Lipschitz continuous and non-decreasing. If further assume that $f'(u) \in L^{\frac{n}{2k}} (\mathbb{H}^n)$, where $u \in W_0^{k,2} (\mathbb{H}^n)$ is a positive solution, then there exists a point $P \in \mathbb{H}^n$ such that $u$ is constant on each geodesic sphere centered at $P$. Moreover, $u$ is radially non-increasing.
\end{thmx}
The moving plane method on $\mathbb{H}^n$ was first seen in the work of Kumaresan-Prajapat  \cites{K-P1,K-P2}, where they established the analog result of Gidas-Ni-Nireneberg type, as well as solved an overdetermined problem on $\mathbb{H}^n$, which is an analogue of the problem on $\mathbb{R}^n$ initiated by Serrin \cite{serrin} (see also \cite{HLZ}, \cite{L-Z2}). The axial symmetry of solutions to a class of integral equations on half spaces was  studied by Lu-Zhu in \cite{L-Z}.  More recently, the overdetermined problem of fractional order equations on hyperbolic spaces was studied by Li-Lu-Wang \cite{LLW2}. The concept of the moving plane method in \cites{K-P1,K-P2} was further developed in the work of Almeida-Ge \cite{ADG2} and Almeida-Damascelli-Ge \cite{ADG1},
where they took advantage of the foliation structure of $\mathbb{H}^n$ (for details see Section \ref{sec2.2}). Such a moving plane method was used earlier in the study of constant mean curvature surface problem by Korevaar, Kusner, Meeks and Solomon \cites{KK,KKMS}. Moreover, as long as there is a certain group of symmetries (either isometries or conformal maps) for the underlying space which also interacts nicely with the given equation, then one can develop these techniques similar to the moving plane method in the Euclidean space. These have been demonstrated nicely by Birindelli and Mazzeo \cite{BirindelliMazzeo}. See also the expository article  \cite{Chow}. Recently. we also obtained a higher order symmetry result of the Brezis-Nirenberg problem on the hyperbolic spaces \cite{LLY} and more general higher order equations \cite{LLW} by using the moving plane method on the hyperbolic space.

The main goal of the present paper is to develop a different approach from those works mentioned above on 
the hyperbolic spaces, namely the moving sphere approach on $\mathbb{H}^n$, to study the symmetry of solutions 
to the equation \eqref{p_k}. The moving sphere method on Euclidean spaces has been used as a variant of 
the moving plane method to study the classification of solutions of partial differential equations. 
This method has produced fruitful results in the past decades, see, for example, Padilla \cite{Padilla}, 
Chen and Li \cite{ChenLi2} and Li and Zhu \cite{Li-Z}. 
In particular, Li and Zhu used the moving sphere method to prove that nontrivial nonnegative solutions of 
some boundary value problems on the upper half space must take a particular form; Chen and Li \cite{ChenLi2} applied the moving sphere method to study the solution of Nirenberg problem; Li \cite{Li} used the moving sphere method to study some conformal invariant integral equations; more recently, Chen, Li and Zhang \cite{CLZ} developed a direct moving sphere method to study fractional order equations without transforming them to integral forms. 
Interested readers can also refer to the book \cite{ChenLi1} (Chapter 8) for a systematic discussion about this method.  

However, as far as we know, such a method has not been established on hyperbolic spaces. One of our main contributions in the present paper is to introduce a Kelvin-type transform on hyperbolic spaces and use it to develop a moving sphere method to study some higher order equations and integral equations. To be precise, we consider the higher order equation in $\hn$:
\begin{equation}\label{p_k}
  P_ku = f(u),
\end{equation}
and our first result reads as follows:
\begin{theorem}\label{thm1}
  Let $1\leq k < \frac{n}{2}$ and $f: [0,\infty)\to\mathbb{R}$ be Lipschitz continuous and nondecreasing such that $f(0)=0, \frac{f(t)}{t^p}$ is non-increasing with $p=\frac{n+2k}{n-2k}$.
  If $u\in L^{\frac{2n}{n-2k}}_{loc}(\hn)$ is a solution of \eqref{p_k} and $f^\prime(u)\in L_{loc}^{\frac{n}{2k}}(\hn)$, then one of the followings holds:
  \begin{enumerate}
    \item $u(x)\equiv C_0$ for some $C_0\geq 0$.
    \item There exists a point $P \in \mathbb{H}^n$ such that $u$ is radially symmetric with respect to $P$. Moreover, $u$ has the form
    \begin{equation}\label{soln}
      u(x)=\frac{\alpha}{(\cosh^2\frac{r(x)}{2}+\beta)^{\frac{n-2k}{2}}},
    \end{equation}
    where $\alpha, \beta$ are some constants, $r(x)$ is the geodesic distance from $x$ to $P$.
  \end{enumerate}
\end{theorem}

By plugging the explicit form of the solution \eqref{soln}, we have the following result immediately.
\begin{corollary}\label{cor1}
  Suppose that $f$ is Lipschitz continuous and nondecreasing such that $f(0)=0$,
  $\frac{f(t)}{t^p}$ is non-increasing with $p=\frac{n+2k}{n-2k}$. If $u$ is a positive nontrivial solution of \eqref{p_k} such that $f^\prime(u)\in L_{loc}^{\frac{n}{2k}}(\hn)$,
  then $f(t)=ct^{\frac{n+2k}{n-2k}}$, for $t\in[0, \max_xu(x)]$.
\end{corollary}

Our method of moving spheres relies on defining the Kelvin transform on the hyperbolic space. The Euclidean Kelvin transform plays an important role when dealing with the asymptotic behavior of $u$ at infinity; see for example \cites{CGS,CLO}. It is also worth mentioning that regarding the nonlocal operator such as the fractional Laplacian,
Chen, Li and Li \cite{CLL} developed a direct method of moving plane together with Kelvin transform
and obtained the following symmetry and nonexistence of positive solutions on $\mathbb{R}^n$.
\begin{thmx}
    Assume that $u\in L_\alpha\cap C^{1,1}_{loc}$ and
    \begin{equation}\label{fl}
        (-\Delta)^{\alpha/2}u(x)=u^p(x), \; x\in\rn,
    \end{equation}
    where the space $L_\alpha$ is defined as

    $$
      L_\alpha = \{ u: \mathbb{R}^n \to \mathbb{R} : \int_{\mathbb{R}^n} \frac{|u(x)|}{1 + |x|^{n + \alpha}} dx < \infty \}.
    $$
    Then,
    \begin{enumerate}
        \item in the subcritical case $1<p<\frac{n+\alpha}{n-\alpha}$, \eqref{fl} has no positive solution.
        \item in the critical case $p=\frac{n+\alpha}{n-\alpha}$, the positive solution of \eqref{fl} must be radially symmetric
              and monotone decreasing about some point in $\rn$.
    \end{enumerate}
\end{thmx}

Now with the Kelvin transform being defined on the hyperbolic space in the present paper, we are able to use the integral version of the moving plane method
to study the $L^{\frac{2n}{n-2k}}_{loc}$ solutions to the equation \eqref{p_k} directly. Our second main result reads as follows:
\begin{theorem}\label{thm2}
    Let $k\geq 2$, $u\in L^{\frac{2n}{n-2k}}_{loc}(\hn)$ be a nonnegative solution of \eqref{p_k}.
    Assume that $f(t)$ is a Lipschitz function satisfying
    \begin{enumerate}[label=(\roman*)]
      \item $f(t)$ is non-decreasing, $ f(0)=0$,
      \item  $f^\prime(u)\in L^{\frac{n}{2k}}_\text{loc}(\hn)$, $f^\prime(t)$ is non-decreasing.
    \end{enumerate}
    Then,
    \begin{enumerate}
        \item in the critical case, $\lim\limits_{t\to\infty}f(t)/t^p=c_1<\infty$ with $p=\frac{n+2k}{n-2k}$, there exists a point $ P \in \mathbb{H}^n$
        such that positive solution $u$ is radially symmetric about some point at $P$. Moreover, $u$ is non-increasing in the radial direction.
        \item in the subcritical case, $\lim\limits_{t\to\infty}f(t)/t^p=c_2<\infty$, where $1<p<\frac{n+2k}{n-2k}$, $u\equiv 0$.
    \end{enumerate}
\end{theorem}

In the remaining part of this section, we shall introduce several applications of Theorem \ref{thm1} regarding the higher order prescribing $Q$-curvature problem on the hyperbolic space.
Let $(M^n,g)$ be an oriented Riemannian $n$-manifold, and $n\geq 3$. One of the central objectives in conformal geometry is whether we can change the metric $g$ conformally into a new metric $h$ with the prescribed curvature function. For example, let $L_g = -\frac{4(n-1)}{(n-2)}\Delta_g+R_g$ be the conformal Laplacian, where $R_g$ is the scalar curvature with respect to $g$.
Setting the conformal factor $\rho=u^{\frac{4}{n-2}}, u>0$ and $h=\rho g$, then we have the following conformal covariance property:
\[L_g(u\varphi) = u^{\frac{n+2}{n-2}}L_h(\varphi), \quad \forall \varphi\in C^\infty(M).\]
If one prescribes the scalar curvature $R_h$ for the metric $h$, then $u$ has to satisfy the second order equation
\begin{equation}\label{Yamabe_eq}
  L_g(u) = u^{\frac{n+2}{n-2}}L_h(1) = R_hu^{\frac{n+2}{n-2}}.
\end{equation}
When $R_h$ is constant, it is the well-known Yamabe problem, and in the case where $R_h$ is a prescribed function, it is called the Nirenberg problem.

If we consider the higher order conformal covariant operators, it relates to the prescribing $Q$-curvature problem.
The $Q$-curvature is a canonically defined $n$-form on $M$. It is not conformally invariant but enjoys certain natural properties with respect to conformal transformations.

In \cites{GOR08}, Grunau, Ahmedou and Reichel considered the fourth order Paneitz equation in hyperbolic space $\hn$ when $n>4$.
They established the existence of a continuum of radially symmetric solutions using ordinary differential equaiton techniques.
The conformal covariance property for Paneitz operators can be stated as follows:
\[P_2^g(u\varphi) = u^{\frac{n+4}{n-4}}P_2^h(\varphi), \quad \forall \varphi\in C^\infty(M).\]
Prescribing the $Q$-curvature for the metric $h$ by the function $Q_h$ leads to the equation
\begin{equation}\label{Qeq_2}
   P_2^g(u) = u ^{\frac{n+4}{n-4}}P_2^h(1) = \frac{n-4}{2}Q_hu^{\frac{n+4}{n-4}}.
\end{equation}
For the hyperbolic space with the Poincar\'e metric $g_{ij}=\frac{4}{(1-|x|^2)^2}\delta_{ij}$, one seeks the metric $h$ of the form $h_{ij}=U^{\frac{4}{n-4}}g_{ij}=u^{\frac{4}{n-4}}\delta_{ij}$,
the equation \eqref{Qeq_2} reduces to 
\begin{equation}\label{ODE_2}
  \begin{cases}
    \begin{aligned}
      \Delta^2u = \frac{n-4}{2}Qu^{\frac{n+4}{n-4}} &\text{ in } B_1(0)\subset\rn,\\
      u > 0 \quad\;\; &\text{ in } B_1, \\
      u = \infty \quad &\text{ on } \partial B_1.
    \end{aligned}
  \end{cases}
\end{equation}
Here the factor $\frac{n-4}{2}$ comes from the expression of the Paneitz operator:
\[P_2^g = \Delta^2_g + \dvg_g(a_nR_g\id - b_n\ric_g)\nabla_g + \frac{n-4}{2}Q_g,\]
where $a_n, b_n$ are some constants depending on $n$.

The authors of \cites{GOR08} proved the following:
\begin{thmx}
  For every $\alpha>0$, there exists a radial solution of the equaiton \eqref{ODE_2} in the unit ball with $Q\equiv \frac{1}{8}n(n^2-4)$,
  infinite boundary values on $\partial B$ and with $u(0)=\alpha$. Moreover,
  \begin{enumerate}
    \item the conformal metric $u^{4/(n-4)}\delta_{ij}$ on $B$ is complete;
    \item if $u(0)>0$ is sufficiently small then the corresponding solution generates a metric with negative scalar curvature.
  \end{enumerate}
\end{thmx}

Their proof is based on a so-called \textit{shooting method} for ODEs. Under the assumption $Q\equiv \frac{1}{8} n(n^2 - 4)$,
the equation \eqref{ODE_2} is equivalent to (up to a scaling)
\begin{equation*}
  \Delta^2 u = u^{\frac{n+4}{n-4}}
\end{equation*}
with the same boundary condition. They focus on the radial solutions by studying the initial value problem 
\begin{equation}
  \begin{cases}
    \Delta^2 u(r) = \left(r^{1-n}\frac{\partial}{\partial r}\left(r^{n-1}\frac{\partial}{\partial r}\right)\right)^2 u(r)= u(r)^{\frac{n+4}{n-4}}, \quad r>0,\\
    u(0) = \alpha,\; u^\prime(0) = 0,\; \Delta u(0) = \beta,\; (\Delta u)^\prime(0) = 0,
  \end{cases}
\end{equation}
where $\alpha\geq 0, \beta\in\mathbb{R}$ are given.
It is known that the above equation has the following explicit positive solution \cites{Lin98,WX99}.
\begin{thmx}
  Suppose that $u$ is a smooth solution of 
  \begin{equation*}
    \begin{cases}
      \begin{aligned}
        &\Delta^ku = u^{\frac{n+2k}{n-2k}}, n>2k,\\
        &u > 0  \text{ in } \rn. 
      \end{aligned}
    \end{cases}
  \end{equation*}
  Then $u$ is radially symmetric about some point $x_0\in\rn$ and $u$ has the following form
  \begin{equation}
    U(x)=\left(\frac{2a}{a^2+r^2}\right)^{\frac{n-2k}{2}}, \quad r=|x-x_0|, \; a>0.
  \end{equation}
\end{thmx}
It is also known that these solutions are the only positive entire solutions \cites{Lin98,WX99}. On the other hand, there exists non-global solution which has the explicit form 
\begin{equation}
  V(x)=\left(\frac{2a}{a^2-r^2}\right)^{\frac{n-4}{2}}, \quad r=|x-x_0|, \; a>0.
\end{equation}
We call such solutions \textit{large radial solutions}, by which we mean they are radially symmetric on open balls that diverge to infinity on the boundary.
In terms of solutions $U$ or $V$, the corresponding metrics are
\begin{equation}
  h=\left(\frac{2a}{a^2\pm r^2}\right)^{2}\delta_{ij}.
\end{equation}
In case of $``+"$ one finds that $(\rn,h)$ is isometrically isomorphic to a sphere $\sn_a$ of radius $a$ 
equipped with standard Euclidian metric scaled by $1/a^2$. 
In case of $``-"$, the solution $U$ blows up on $\partial B_a(x_0)$ and one finds that $(B_a(x_0),h)$ 
is isometrically isomorphic to the standard hyperbolic space.

Motivated by the results in \cites{GOR08}, it is natural to consider the problem related to the GJMS operators of order greater than 4 in the hyperbolic space $\hn$ for $n>2k$.
With the chosen conformal factor $\rho=u^{\frac{2k}{n-2k}}$ and $h=\rho g$, we have the conformal covariance property as follows:
\[P_k^g(u\varphi) = u^{\frac{n+2k}{n-2k}}P_k^h(\varphi), \quad \forall \varphi\in C^\infty(M).\]
The prescribing $Q$-curvature problems lead to considering the equation
\begin{equation}\label{Qeq_k}
  P_k u = \frac{n-2k}{2}Qu^{\frac{n+2k}{n-2k}}.
\end{equation}
In fact, due to the establishment of the explicit form of solution \eqref{soln}, we are able to answer the prescribing constant $Q$-curvature problem on $\hn$ affirmatively.
\begin{theorem}\label{PrescribedH}
  There exists a positive solution to the prescribed constant $Q$-curvature problem \eqref{Qeq_k} on the hyperbolic space.
\end{theorem}
We point out that our study of the equation on the hyperbolic space and the classification result provide fruitful information to study the corresponding Euclidean problem. In fact, there are plenty of efforts devoted to the large radial solutions of the following polyharmonic equation with  nonlinearity of polynomial growth, for example \cites{DLS09, DLS14}.
\begin{equation}\label{radial_eq_eucl}
  \begin{cases}
    \begin{aligned}
      &(-\Delta)^k u(r) = \left(r^{1-n}\frac{\partial}{\partial r}\left(r^{n-1}\frac{\partial}{\partial r}\right)\right)^k u(r)= u(r)^{\frac{n+2k}{n-2k}}, \\
      & u(0)=\alpha, u^{\prime}(0)=0,\\
      &(-\Delta)^m u(0) = \beta_m,\; (-\Delta)^m u^\prime(0) = 0, \quad m=1,\cdots,k-1. 
    \end{aligned}
  \end{cases}
\end{equation}
We will see that the solution $U$ is a separatrix in the $r-u$ plane, i.e., if we fix $\alpha > 0$ and consider $\beta$ as a varying parameter then $U_\alpha$ separates the large solutions from the solutions with entire solutions. Denote $U(0)=\alpha_0$ and $\Delta U(0)=\beta$, we have
\begin{theorem}\label{thm3}
  For any $\alpha>0$, if $u$ is a positive radial solution of \eqref{radial_eq_eucl}, then $\beta_1\leq\beta$. Moreover,
  \begin{enumerate}
    \item if $\beta_1<\beta$, then $u$ blows up on $\partial B_1(0)$, i.e., $u(r)\to\infty$ as $r\to 1$.
    \item if $\beta_1=\beta$, then $u$ takes the form \[u(r)=\left(\frac{2a}{a^2+r^2}\right)^{\frac{n}{2}-k}.\]
  \end{enumerate} 
\end{theorem}

Clearly, if $u\equiv 1$, we are at the Poincar\'e metric, which gives the constant $Q$-curvature. For nontrivial positive solutions, they correspond to the blow-up solutions of the equation \eqref{radial_eq_eucl} in $\rn$.

The authors of \cite{GOR08} also considered the radial solution of \eqref{ODE_2} for a prescribed smooth radial $Q$-curvature function using the shooting method. They provided the necessary conditions of the nonconstant $Q$-curvature in order to have the radial solution. More precisely, they showed that there exists a radial solution of 
\begin{equation}
  \begin{cases}
    \begin{aligned}
      \Delta^2u = \tilde{Q}u^{\frac{n+4}{n-4}} &\text{ in } B_1(0)\subset\rn,\\
      u > 0 \quad\;\; &\text{ in } B_1, \\
      u = \infty \quad &\text{ on } \partial B_1.
    \end{aligned}
  \end{cases}
\end{equation}
if the function $\tilde{Q}: B_1(0)\to\mathbb{R}$ satisfies 
\begin{enumerate}
  \item there are two positive constants $Q_0, Q_1$ such that $0<Q_0<\tilde{Q}(r)<Q_1$ on $[0,1], \tilde{Q}\in C^1[0.1]$,
  \item there exists $q\in[0,1)$ such that $r^q\tilde{Q}(r)$ is nondecreasing.
\end{enumerate}

Now if we consider the problem
\begin{equation}\label{Qeq_k_nonconstant}
  P_k u = \tilde{Q}u^{\frac{n+2k}{n-2k}}.
\end{equation}
where $\tilde{Q}$ is a function on $\hn$. Our classification result also provide exact information on the right hand side, we then obtain the following result as a consequence of Corollary \ref{cor1}.
\begin{corollary}\label{cor2}
   Suppose that $\tilde{Q}$ is decreasing such that $\tilde{Q}\in L^\infty_{loc}(\hn)$. If there exists a radial positive solution $u\in L^{\frac{2n}{n-2k}}_{loc}\cap W^{k,2}(\hn)$ of the equation \eqref{Qeq_k_nonconstant} on $\hn$, then $\tilde{Q}$ must be constant.
\end{corollary}

The following remarks are in order. The prescribed $Q$-curvature problem we consider in this paper is on the noncompact hyperbolic space $\mathbb{H}^n$.  In contrast to the large literature on the existence of solutions to the prescribed Q-curvature problem on compact Riemannian manifolds, see for example Chang and Yang \cite{ChangYang-Annals}, Gursky \cite{Gursky}, Wei and Xu \cite{WX09}, Djadli and Malchiodi \cite{DM08}, Robert \cite{Robert}, Djadli, Hebey and Ledoux \cite{DHL}, (see also Chang \cite{Chang} and Hang and Yang \cite{HY16}), there has been very little study in our noncompact setting on the entire hyperbolic space other than the work of \cite{GOR08}. In \cites{GOR08}, Grunau, Ahmedou and Reichel considered the fourth order Paneitz equation in hyperbolic space $\hn$ when $n>4$. Our Theorem \ref{PrescribedH} and Corollary \ref{cor2} offer some new understandings towards this problem. 

The organization of this paper is as follows. In Section \ref{sec:prelim}, we provide some preliminaries on Helgason-Fourier analysis, foliations and key inequalities on hyperbolic spaces. In Section \ref{sec:Kelvin}, we introduce the precise definition of our Kelvin transforms on the hyperbolic spaces. Section \ref{sec:proofs} is dedicated to the proofs of main theorems and corollaries.

\section{Notations and Preliminaries}\label{sec:prelim}
In this section we present some preliminaries concerning hyperbolic space, the Helgason-Fourier analysis theory, Hardy-Littlewood inequalities and GJMS operators,  which are needed in the sequel.
\subsection{Models of hyperbolic spaces}\label{sec2.1}
The hyperbolic space $\hn$ $(n\geq 2)$ is a complete, simply connected Riemannian manifold with constant sectional curvature $-1$.
There are several analytic models of hyperbolic spaces, all of which are equivalent.
Among them, we describe two models here.

\begin{itemize}
    \item \textit{The Half-space model}:
    It is given by $\mathbb{R}^{n-1}\times \mathbb{R}^+=\{(x_1,\cdots,x_{n-1},x_n):x_n>0\}$,
    equipped with the Riemannian metric
    $$ds^2=\frac{dx_1^2+\cdots+dx_n^2}{x_n^2}.$$
    The hyperbolic volume element is $dV = \frac{dx}{x_n^n}$, where $dx$ is the Lebesgue measure on $\mathbb{R}^n$. The hyperbolic gradient is $\nabla_\mathbb{H} = x_n \nabla$ and the Laplace-Beltrami operator on $\mathbb{H}^n$ is given by

$$
  \Delta\mathbb{H} = x_n^2 \Delta - (n-2) x_n \frac{\partial}{\partial x_n},
$$
where $\Delta$ is the usual Laplacian on $\mathbb{R}^n$.

    \item \textit{The Poincar\'e ball model}:
    It is given by the open unit ball $\mathbb{B}^n=\{x=(x_1,\cdots,x_n):x_1^2+\cdots+x_n^2<1\}\in\mathbb{R}^n$ equipped with the Poincar\'e metric
    $$ds^2=\frac{4\left(dx_1^2+\cdots+dx_n^2\right)}{\left(1-|x|^2\right)^2}.$$
    The distance from $x\in\mathbb{B}^n$ to the origin is $\rho(x,0)=\log\frac{1+|x|}{1-|x|}$.
    The hyperbolic volume element is $dV=\left(\frac{2}{1-|x|^n}\right)^2dx$.
    The hyperbolic gradient is $\nabla_{\mathbb{H}^n}=\frac{1-|x|^2}{2}\nabla$ and the associated Laplace-Beltrami operator is given by
    $$\Delta_{\mathbb{H}}=\frac{1-|x|^2}{4}\left((1-|x|^2)\Delta+2(n-2)\sum_{i=1}^nx_i\frac{\partial}{\partial x_i}\right).$$
\end{itemize}

\subsection{Foliations of hyperbolic spaces}\label{sec2.2}
In this subsection, we introduce the foliation of $\mathbb{H}^n$ which is needed in the study of the hyperbolic 
symmetry of the solution (see \cites{ADG1,ADG2} for more detail). Let $\mathbb{R}^{n,1} = ( \mathbb{R}^{n+1} , g )$ 
be the Minkowski space, where the metric $ds^2 = -dx_0^2 + dx_1^2 + \cdots + dx_n^2$. 
The hyperboloid model of hyperbolic space $\mathbb{H}^n$ is the submanifold $\{ x \in \mathbb{R}^{n,1}: g(x,x) = -1, x_0 > 0 \}$. A particular directional foliation can be obtained by choosing any direction in the $x_1, x_2, \cdots , x_n$ plane. Without loss of generality, we consider $x_1$ direction. Denote $\mathbb{R}^{n,1} = \mathbb{R}^{1,1} \times \mathbb{R}^{n-1}$, where $(x_1, x_0) \in \mathbb{R}^{1,1}$. Define $A_t = \tilde{A}_t \otimes Id_{\mathbb{R}^{n-1}}$, where $\tilde{A}_t$ is the hyperbolic rotation on $\mathbb{R}^{1,1}$:
$$
  \tilde{A}_t = \left(\begin{array}{cc}\cosh t & \sinh t \\\sinh t & \cosh t \end{array}\right).
$$
The reflection is defined by $I (x_0, x_1, x_2,  \cdots, x_n) = (x_0, -x_1, x_2, \cdots, x_n)$. Let $U = \mathbb{H}^n \cap \{ x_1 = 0 \}$ and $U_t = A_t (U)$, then $\mathbb{H}^n$ is foliated by $U_t$, i.e. $\mathbb{H}^n = \cup_{t \in \mathbb{R}} U_t$. Moreover, if we define $I_t = A_t \circ I \circ A_{-t}$, then it is easy to verify that $I_t (U_t) = U_t$.

\subsection{The Helgason-Fourier transform on hyperbolic spaces}

In this subsection, we will introduce the Helgason-Fourier transform on hyperbolic spaces and 
define the fractional Laplacian on hyperbolic spaces using the Helgason-Fourier analysis. 
For complete details, we refer to \cites{H-F1,H-F2}. Let
$$e_{\lambda,\zeta}(x)=\left(\frac{\sqrt{1-|x|^2}}{|x-\zeta|}\right)^{n-1+i\lambda}, \; x\in\mathbb{B}^n, \lambda\in\mathbb{R}, \zeta\in\mathbb{S}^{n-1}.$$
Then the Fourier transform of a function $f$ on $\mathbb{H}^n$ (ball model) is defined as
$$\hat{f}(\lambda,\zeta)=\int_{\bn}f(x)e_{-\lambda,\zeta}(x)dV,$$
provided the integral exists. Moreover, the following inversion formula holds for $f\in C^\infty_0(\bn)$:
$$f(x)=D_n\int_{-\infty}^\infty\int_{\mathbb{S}^{n-1}}\hat{f}(\lambda,\zeta)e_{\lambda,\zeta}(x)|\mathfrak{c}(\lambda)|^{-2}d\lambda d\sigma,$$
where $D_n=(2^{3-n}\pi|\mathbb{S}^{n-1}|)^{-1}$ and $\mathfrak{c}(\lambda)$ is the Harish-Chandra's $\mathfrak{c}$-function given by
$$\mathfrak{c}(\lambda)=\frac{2^{n-1-i\lambda}\Gamma(n/2)\Gamma(i\lambda)}{\Gamma(\frac{n-1+i\lambda}{2})\Gamma(\frac{1+i\lambda}{2})}.$$
There also holds the Plancherel formula:
$$\int_{\bn}|f(x)|^2dV=D_n\int_{-\infty}^\infty\int_{\mathbb{S}^{n-1}}\hat{f}(\lambda,\zeta)|\mathfrak{c}(\lambda)|^{-2}d\lambda d\sigma.$$

\subsection{Hardy-Littlewood-Sobolev inequality}

The Hardy-Littlewood-Sobolev inequality inequality on hyperbolic ball model $\bn$ is equivalent to the one on the hyperbolic upper half spaces model.
It was first proved on half spaces by Beckner \cite{Beckner}, and then on Poincar\'e ball by Lu and Yang \cite{LY1}.

\textbf{Theorem.} \textit{Let $0<\lambda<n$ and $p=\frac{2n }{2n-\lambda}$. Then for $f,g\in L^p(\bn)$,}
\begin{equation}\label{HLS}
    \left|\int_{\bn}\int_{\bn}\frac{f(x)g(y)}{(2\sinh (\frac{\rho(x,y)}{2}))^\lambda}dV_xdV_y\right|\leq C_{n,\lambda}\|f\|_p\|g\|_p,
\end{equation}
\textit{where $\rho(x,y)$ denotes the geodesic distance between $x$ and $y$ and }
\begin{equation*}
    C_{n,\lambda}=\pi^{\lambda/2}\frac{\Gamma(\frac{n}{2}-\frac{\lambda}{2})}{\Gamma(n-\frac{\lambda}{2})}\left(\frac{\Gamma(\frac{n}{2})}{\Gamma(n)}\right)^{-1+\frac{\lambda}{n}}
\end{equation*}
\textit{is the best constant for the classical Hardy-Littlewood-Sobolev constant on $\mathbb{R}^n$.
Furthermore, the constant $C_{n,\lambda}$ is sharp and there is no nonzero extremal function for the inequality (\ref{HLS}).}

\subsection{GJMS operator and its fundamental solution}

GJMS operators were discovered in the work of Graham, Jenne, Mason, and Sparling \cite{GJMS2} based on
the construction of ambient metric by C. Fefferman and Graham \cite{FeffermanGr}.

A differential operator $D$ is conformally covariant of bidegree $(a,b)\in\mathbb{R}^2$ in dimension $n$ if for any
$n$-dimensional Riemannian manifold $(M,g_0)$,
\begin{equation*}
  g_\omega=e^{2\omega}g_0, \omega\in C^\infty(M) \implies D_\omega=e^{-b\omega}D_0\mu(e^{a\omega}),
\end{equation*}
where for any $f\in C^\infty(M), \mu(f)$ is multiplication by $f$. Here the subscripts indicate the corresponding metric.
It was shown in \cite{GJMS2} and \cite{Bra} that GJMS operators are
conformally covariant differential operators of bidegree $(\frac{n}{2}-k,\frac{n}{2}+k)$.

Recently, it has been shown by Lu and Yang \cite{LY2}, that the fundamental solution of $P_k$ on $\hn$ has the following explicit expression:
\begin{equation*}
 P_k^{-1}(\rho)=\frac{\Gamma(\frac{n}{2})}{2^n\pi^{\frac{n}{2}}\Gamma(k)\Gamma(k+1)}\frac{\left(\cosh\frac{\rho}{2}\right)^{-n}}{\left(\sinh\frac{\rho}{2}\right)^{n-2k}}F\left(k-\frac{n-2}{2},k;k+1;\cosh^{-2}\frac{\rho}{2}\right),
\end{equation*}
where $F$ is the hypergeomtric function. It is immediate to see that for any fixed $y\in \mathbb{B}^n, G(x,y)=P_k^{-1}(\rho)$ is a positive radially decreasing function with respect
to the geodesic distance $\rho=\rho(x,y)$.
Furthermore, we have, for $1\leq k < \frac{n}{2}$,
\begin{equation}\label{pk_est}
P_k^{-1}(\rho)\leq\frac{1}{\gamma_n(2k)}\left[\left(\frac{1}{2\sinh\frac{\rho}{2}}\right)^{n-2k}-\left(\frac{1}{2\cosh\frac{\rho}{2}}\right)^{n-2k}\right], \;\rho>0.
\end{equation}

\section{Kelvin Transform on Hyperbolic Spaces \texorpdfstring{$\hn$}{}}\label{sec:Kelvin}

We first recall that on Euclidean spaces spherical inversion $i_{a,\lambda}$ with respect to the sphere $\partial B_\lambda (x_0)$ is defined as follows:
\begin{align}\label{inv_rn}
  i_{x_0,\lambda}:\quad &\rn\setminus\{ x_0 \} \rightarrow \rn\setminus\{x_0\}\nonumber\\
&x\mapsto x_0+\frac{\lambda^2(x-x_0)}{|x-x_0|^2}.
\end{align}
Here we will introduce a Kelvin type transform on the hyperbolic space $\hn$.
Consider a geodesic ball $B_\lambda(P) \subset \mathbb{H}^n$ centered at $P$ with radius $\lambda$, we want to obtain a partially defined conformal mapping $\phi_{P,\lambda}$ which fixes the boundary $\partial B_\lambda(p)$ and maps the interior of the geodesic ball to the exterior, and vice versa.

It is known that the Riemannian metric on $\hn$ can be written as $g=dr^2+\sinh^2(r)d\theta^2$. For any point $x=(r,\theta)\in\hn$
in geodesic polar coordinates centered at $P$, we denote
\[x^\lambda:=\phi_{P,\lambda}(r,\theta) = (\varphi_\lambda(r),\theta).\]
Then the new metric under the inversion map is
$\tilde{g}=(\varphi_\lambda(r)')^2 dr^2+\sinh^2(\varphi_\lambda (r))d\theta^2$.
Since $\tilde{g}$ is conformal to $g$, we obtain the following ODE:
\begin{equation}
  \begin{cases}
    &\varphi_\lambda^{\prime}(r)=-\frac{\sinh \varphi_\lambda(r)}{\sinh r},\\
    &\varphi_\lambda(\lambda)=\lambda.
  \end{cases}
\end{equation}
Integrating both sides, we obtain
\begin{align*}
  \int\frac{1}{\sinh \varphi_\lambda(r)}dh_\lambda(r)=-\int\frac{1}{\sinh r}dr,
\end{align*}
so that
\begin{equation*}
  \log(\tanh\frac{\varphi_\lambda(r)}{2})=\log\frac{C}{\tanh\frac{r}{2}}
\end{equation*}
subject to the boundary condition
\begin{equation*}
  \log(\tanh\frac{\lambda}{2})=\log\frac{C}{\tanh\frac{\lambda}{2}}.
\end{equation*}
We thus obtain the solution:
\[\varphi_\lambda(r)=2\arctanh\left[\left(\tanh\frac{\lambda}{2}\right)^{2}\left(\tanh\frac{r}{2}\right)^{-1}\right]\]
The Jacobian of  $\phi_{P,\lambda}$ is
\begin{equation}\label{Jacobian}
  |J_{\phi_{P,\lambda}}(r,\theta)|=\det\left(\frac{\partial x^\lambda_i}{\partial x_j}\right)=\left(\frac{(\tanh\frac{\lambda}{2})^2(\csch \frac{r}{2})^{2}}{1-(\tanh\frac{\lambda}{2})^4(\coth\frac{r}{2})^{2}}\right)^n.
\end{equation}
Due to the fact that the hyperbolic space $\hn$ is merely conformal to a unit ball $\bn\subset\rn$,
 $\varphi_\lambda(r)$ is only defined for $r>2\arctanh(\tanh^2\frac{\lambda}{2})$.
The conformal infinity as the boundary of the Poincar\'e ball is mapped to the set, which we call the limit sphere
\[S_{\lambda^\sharp}(P)=\partial B_{\lambda^\sharp}(P)=\{x\mid \rho(x,P)=\lambda^\sharp:=2\arctanh(\tanh^2\frac{\lambda}{2})\},\]
under the inversion. Now we define the Kelvin transform of $u$ with respect to the geodesic sphere centered at $P$ with radius $\lambda$ as
\begin{equation}\label{Kelvin}
  u_{P,\lambda}(x) = |J_{\phi_{P,\lambda}}|^{\frac{n-2k}{2n}}u(x^\lambda),\; \forall x\in\hn\setminus B_{\lambda^\sharp}(P).
\end{equation}

\section{Proofs}\label{sec:proofs}

First, we show that the differential equation (\ref{p_k}) is equivalent to an integral equation using the Helgason-Fourier transform on hyperbolic spaces.
\begin{lemma}\label{lemma1}
  If $u\in W^{k,2}_0(\hn)$ is a positive weak solution of the higher order
  differential equation (\ref{p_k}), then under the assumptions of $f$ in Theorem \ref{thm1}, $u$ must satisfy the following integral equation for the
  Green's function $G(x, y)$ of the GJMS operator $P_k$ on the hyperbolic ball $\mathbb{B}^n$:
  \begin{equation}\label{inteq0}
      u(x)=\int_{\mathbb{B}^n} G(x,y)f(u(y))dV_y.
  \end{equation}
\end{lemma}

\begin{proof}
  $u\in W^{k,2}_0(\hn)$ is a weak solution of \eqref{inteq0} if and only if for any $\phi\in C^\infty_0(\hn)$,
  $$\inth \sum_{j=1}^kc_j\nabla^j_{\hn}u\nabla^j_{\hn}\phi = \inth f(u)\phi dV,$$
  where $c_j$'s are some constant depending on $P_k$.
  Using Helgason-Fourier transform on $\hn$, such definition of weak solution is equivalent to
  \begin{equation*}
      D_n\int_{-\infty}^{+\infty}\int_{\mathbb{S}^{n-1}}\left[\prod_{j=1}^k\left(\frac{\tau^2+(2k-1)^2}{4}\right)\right]\hat{u}(\tau,\sigma)\hat{\phi}(\tau,\sigma)|\mathfrak{c}(\tau)|^{-2}d\sigma d\tau = \inth f(u)\phi dV.
  \end{equation*}
  Let $\psi$ satisfy $P_k\psi = \phi$, that is, $\psi(x)=\inth G(x, y)\phi(y)dV_y$. Under Helgason-Fourier transform, we have
  \begin{equation*}
      \hat{\phi}(\tau,\sigma)=\prod_{j=1}^k\frac{\tau^2+(2k-1)^2}{4}\hat{\psi}(\tau,\sigma).
  \end{equation*}
  Now if we replace $\phi$ by $\psi$, we get
  \begin{equation*}
      D_n\int_{-\infty}^{+\infty}\int_{\mathbb{S}^{n-1}}\hat{u}(\tau,\sigma)\hat{\psi}(\tau,\sigma)|c(\tau)|^{-2}d\sigma d\tau = \inth f(u)\left(\inth G(x,y)\phi(y)dV_y\right)dV_x.
  \end{equation*}
  Applying Plancherel formula to the left hand side, we get
  \begin{equation*}
      \inth u(x)\phi(x)dV = \inth\left(\inth G(x,y)f(x)dV_x\right)\phi(y)dV_y,
  \end{equation*}
  which is true for any $\phi\in\cs(\hn)$. This immediately implies that a solution of the differential equation (\ref{p_k})
  is a solution of the integral equation (\ref{inteq0}).
    \end{proof}

\subsection{Proof of Theorem \ref{thm1}}
For simplicity, we will drop the subscript $P$ whenever it causes no confusion.
Now that GJMS operators are conformally covariant with bidegree $(\frac{n}{2}-k,\frac{n}{2}+k)$, under the Kelvin transform we defined in Section \ref{sec:Kelvin}, we have:
\begin{equation}\label{pk_conformal}
  P_k u_{\lambda} = |J_{\phi_{\lambda}}|^{\frac{n+2k}{2n}}(P_k u) \circ \phi_{\lambda}.
\end{equation}
Denote
  \[w_\lambda(x)= u_{\lambda}(x) - u(x).\]
One immediately sees that $w_\lambda$ is anti-symmetric since the Kelvin transform is its own inverse.
We next observe that
  \begin{equation}
    P_kw_\lambda(x)+c(x)w_\lambda(x)=0
  \end{equation}
holds for some $c(x)$ depending on  $u(x)$.
Now if we define $f_\lambda(u(x)):=f(u(x^\lambda))$, we have
  \[P_ku_{\lambda}(x) = |J_{\phi_{\lambda}}|^{\frac{n+2k}{2n}}(P_ku)\circ\phi_\lambda=|J_{\phi_{\lambda}}|^{\frac{n+2k}{2n}}f_\lambda(u(x))\]
when $x\in B_{r_0}^c(P)$.
As discussed in Section \ref{sec:Kelvin}, $x^\lambda$ approaches $\partial B_{\lambda^\sharp}(P)$ from the exterior as $x$ approaches infinity,
and $x^\lambda$ is undefined for $x$ inside $B_{\lambda^\sharp}$.

Similar to procedures of the method of moving plane, we need the following property near the initial position. With a bit of abuse of notation, we will still denote the radii of the reflection spheres $\lambda$ as we increase them from $\lambda^\sharp$.
\begin{lemma}\label{start1}
  For sufficiently small $\lambda=\lambda^\sharp+\varepsilon$, where $\varepsilon>0$, there holds,
  \begin{equation}\label{ineq0}
    w_\lambda(x)\leq 0,\; \forall x\in B^c_\lambda(P).
  \end{equation}
\end{lemma}

\begin{proof}
Define
\begin{equation}
\Sigma_\lambda^-=\{x\in B^c_\lambda(P):w_\lambda(x)>0\}.
\end{equation}
In order to prove the lemma, it suffices to show that $\Sigma_\lambda^-$ has zero measure for sufficiently small $\varepsilon$.
Since we have that for any $x \in B^c_{\lambda^\sharp} \subset \hn$ where $\lambda^\sharp = 2 \operatorname{arctanh}(\tanh^2 \frac{\lambda}{2})$,
\begin{align*}
u_\lambda(x) &= |J_\lambda(x)|^{\frac{n-2k}{2n}} u(x^\lambda) = \int_{\hn} |J_\lambda(x)|^{\frac{n-2k}{2n}} G(x^\lambda, y) f(u)(y) dV_y.
\end{align*}
Noting the decomposition
\[\hn\setminus B_{\lambda^\sharp}(P)=\phi(B^c_\lambda(P))\cup B^c_\lambda(P),\]
we then have that for any $x \in \Sigma^-_\lambda = \{x \in B^c_\lambda: u_\lambda(x) > u(x)\}$
\begin{align}
u_\lambda(x) &- u(x) = \int_{\hn} |J_\lambda(x)|^{\frac{n-2k}{2n}} G(x^\lambda, y) f(u)(y) dV_y - \int_{\hn} G(x,y) f(u)(y) dV_y \nonumber\\
= &\int_{B_{\lambda^\sharp}} (|J_\lambda(x)|^{\frac{n-2k}{2n}} G(x^\lambda, y) - G(x,y)) f(u)(y) dV_y  \label{control_on_B_r0} \\
+ &\int_{B_\lambda \backslash B_{\lambda^\sharp}} (|J_\lambda(x)|^{\frac{n-2k}{2n}} G(x^\lambda, y) - G(x,y)) f(u)(y) dV_y \label{control_on_B1} \\
+ &\int_{B^c_{\lambda}} (|J_\lambda(x)|^{\frac{n-2k}{2n}} G(x^\lambda, y) - G(x,y)) f(u)(y) dV_y \label{int_out}
\end{align}

First by noticing $\rho(x,y)>\rho(x^\lambda,y)$ so that $G(x^\lambda, y) > G(x,y)$ for $x\in\Sigma_\lambda^-$ and $y\in B_{\lambda^\sharp}(P)$, the Green's function estimate of $P_k$ implies that
\begin{equation*}
G(x^\lambda,y)\lesssim \frac{1}{\sinh(\lambda-\lambda^\sharp)^{n-2k}}\sim\frac{1}{(\lambda-\lambda^\sharp)^{n-2k}} \text{ as } \varepsilon\to 0.
\end{equation*}
Together with Lebesgue differentiation theorem, we obtain that
\begin{align}
&\int_{B_{\lambda^\sharp}(P)} G(x^\lambda,y)f(u(y))dV_y \lesssim \frac{1}{(\lambda-\lambda^\sharp)^{n-2k}}\int_{B_{\lambda^\sharp}(P)}f(u(y))dV_y\nonumber\\
&\longrightarrow \frac{vol(B_{\lambda^\sharp})}{(\lambda-\lambda^\sharp)^{n-2k}}f(u(P)), \text{ as } \varepsilon\to 0.
\end{align}
Recall that the hyperbolic volume $vol(B_{r})=\frac{2\pi^{n/2}}{\Gamma(n/2)}\int_0^{r}\sinh^{n-1}tdt$ and $\sinh t \sim t$ when $t$ small,
so that $vol(B_{\lambda^\sharp})\sim (\lambda^\sharp)^n$ and $\lim\limits_{\lambda\to 0} \frac{vol(B_{\lambda^\sharp})}{(\lambda-\lambda^\sharp)^{n-2k}} = 0$.
Thus, when $\lambda$ is small enough, we conclude that \eqref{control_on_B_r0} converges to zero. 

We next claim the following,
\begin{equation}\label{G_relation}
|J_\lambda(y)|^{\frac{n-2k}{2n}} G(x, y^\lambda) = |J_\lambda(x)|^{\frac{n-2k}{2n}} G(x^\lambda, y).
\end{equation}  
To see this, consider any cut-off function $\eta \in C_0^\infty(\hn\setminus B_{r_0})$, we have
\begin{align*}
&\int_{\hn \setminus B_{r_0}} |J_\lambda(y)|^{\frac{n-2k}{2n}} G(x, y^\lambda) \eta(y) dV_y = \int_{\hn \setminus B_{r_0}} |J_\lambda(y)|^{-\frac{n-2k}{2n}} G(x,y) \eta(y^\lambda) |J_\lambda(y)| dV_y \\
&= \int_{\hn} G(x,y) |J_\lambda(y)|^{\frac{n+2k}{2n}} \eta(y^\lambda) dV_y = P_k^{-1}(\eta_\lambda).
\end{align*}
On the other hand, we have that 
\begin{align*}
&\int_{\hn \setminus B_{r_0}} |J_\lambda(x)|^{\frac{n-2k}{2n}} G(x^\lambda, y) \eta(y) dV_y = |J_\lambda(x)|^{\frac{n-2k}{2n}} \int_{\hn} G(x^\lambda, y) \eta(y) dV_y \\
&= |J_\lambda(x)|^{\frac{n-2k}{2n}} (P_k^{-1} \eta)(x^\lambda)
\end{align*}

By the conformal property of $P_k$, we have
\[
P_k^{-1}(\phi_\lambda) = |J_\lambda(x)|^{\frac{n-2k}{2n}} (P_k^{-1} \phi)(x^\lambda).
\]
Hence we showed that
\[
|J_\lambda(y)|^{\frac{n-2k}{2n}} G(x, y^\lambda) = |J_\lambda(x)|^{\frac{n-2k}{2n}} G(x^\lambda, y).
\]
By substituting $y$ with $y^\lambda$, we further have
\[
|J_\lambda(x)|^{\frac{n-2k}{2n}} |J_\lambda(y)|^{\frac{n-2k}{2n}} G(x^\lambda, y^\lambda) = G(x,y),
\]
Now for the integral in \eqref{control_on_B1}, it equals
\begin{align}
  &\int_{B^c_{\lambda}} (|J_\lambda(x)|^{\frac{n-2k}{2n}} G(x^\lambda, y^\lambda) - G(x, y^\lambda)) f(u)(y^\lambda) |J_\lambda(y)| dV_y \nonumber\\
  &=\int_{B^c_{\lambda}}\left(|J_\lambda(x)|^{\frac{n-2k}{2n}}|J_\lambda(y)|^{\frac{n-2k}{2n}}G(x^\lambda,y^\lambda)-|J_\lambda(y)|^\frac{n-2k}{2n}G(x,y^\lambda)\right)|J_\lambda(y)|^{\frac{n+2k}{2n}}f(u)(y^\lambda)dV_y\nonumber\\
  &=\int_{B^c_{\lambda}}\left(G(x,y)-|J_\lambda(y)|^\frac{n-2k}{2n}G(x,y^\lambda)\right)|J_\lambda(y)|^{\frac{n+2k}{2n}}f_\lambda(u(y))dV_y.
\end{align}
For the integral \eqref{int_out}, using the indentity \eqref{G_relation}, we have
\begin{align}
  &\int_{B^c_{\lambda}} (|J_\lambda(x)|^{\frac{n-2k}{2n}} G(x^\lambda, y) - G(x,y)) f(u)(y) dV_y \nonumber\\
  &=\int_{B^c_{\lambda}} (|J_\lambda(y)|^{\frac{n-2k}{2n}} G(x, y^\lambda) - G(x,y)) f(u)(y) dV_y
\end{align}
Therefore, we obtain that
\begin{align*}
u_\lambda(x) &- u(x) =\int_{B^c_\lambda(P)}K(x,y)\left(|J_{\phi_\lambda} (y)|^{\frac{n+2k}{2n}}f_\lambda(u(y)) - f(u(y))\right)dV_y,
\end{align*}
where
\begin{equation}
K(x,y) = 
 G(x,y) - |J_{\phi_\lambda}(y)|^{\frac{n-2k}{2n}}G(x,y^\lambda).
\end{equation}

  Now regarding the kernel $K(x,y)$, we make the conformal change from $(\bn,\frac{4dx^2}{(1-|x^2|)^2})$ to $(\bn, dx^2)$,
  and get the corresponding kernel $\tilde{K}(x,y)=G_{\rn}(x,y)-|J_{\tilde{\phi}_\lambda}|^{\frac{n-2k}{2n}}G_{\rn}(x,y^\lambda)$,
  where $G_{\rn}(x,y)$ stands for the Green's function of $(-\Delta)^{k}$ in $\rn$ and $J_{\tilde{\phi}_\lambda}$ the
  Jacobian of the inversion \eqref{inv_rn}. In fact, we have the following,
  \begin{equation*}
    \tilde{K}(x,y)=\frac{1}{|x-y|^{n-2k}}-\left(\frac{\lambda}{|y-P|}\right)^{n-2k}\frac{1}{|x-y^\lambda|^{n-2k}}, \quad \forall x,y\in B^c_\lambda(P)\subset\rn.
  \end{equation*}
  The kernel $\tilde{K}$ here is equivalent to the positive kernel in the proof of Theorem 1 of \cite{LYY}. To see this, without loss of generality,
  we assume $P=0$ is the origin and notice the Equation (30) in \cite{LYY}:
  \[\frac{|x||y|}{\lambda^2}|x^\lambda-y^\lambda|=|x-y|.\]
  Direct calculation gives
  \[\frac{\lambda}{|y|}\frac{1}{|x-y^\lambda|}<\frac{1}{|x-y|}.\]
  Therefore, we obtain that
  \begin{equation}\label{K>0}
    \tilde{K}(x,y) > 0, \text{ for any } x,y\in B^c_\lambda(P),
  \end{equation}
  and so is $K(x,y)$ in $\hn$.

  Now when $x\in\Sigma^-_\lambda$, we have
  \begin{align}
    0<w_\lambda(x) &= \int_{B^c_\lambda(P)}K(x,y)\left(|J_{\phi_\lambda} (y)|^{\frac{n+2k}{2n}}f(u(y^\lambda)) - f(u(y))\right)dV_y\nonumber\\
    &\lesssim \int_{B^c_\lambda(P)}K(x,y)\left(f(u_\lambda(y)) - f(u(y))\right)dV_y,
  \end{align}
  where in the last line we use the growth assumption of $f$. Moreover, by the definition of $K(x,y)$ and Green's function estimate for $P_k$,
  \begin{align*}
    w_\lambda(x)
    &\lesssim \int_{\Sigma_\lambda^-}G(x,y)\left(f(u_\lambda) - f(u)\right)dV_y\\
     &\leq \int_{\Sigma_\lambda^-}\left(\frac{1}{2\sinh\frac{\rho(x,y)}{2}}\right)^{n-2k}\left(f(u_\lambda)-f(u)\right)dV_y
  \end{align*}

  Following the Hardy-Littlewood-Sobolev inequality on $\hn$, we get
  \begin{align}\label{w_est}
    \|w_\lambda(x)\|_{L^q(\Sigma_\lambda^-)}\leq C\|f(u_\lambda)-f(u)\|_{L^{\frac{nq}{n+2kq}}(\Sigma^-_\lambda)}.
  \end{align}
  Since $f$ is Lipschitz continuous and nondecreasing, it follows the mean value theorem that in $\Sigma^-_\lambda$,
  \[f(u_\lambda)-f(u)\leq Cf^\prime(u_\lambda)(u_\lambda-u). \]
  Applying H\"older's inequality for \eqref{w_est}, we obtain that, for some $q>\frac{n}{n-2k}$,
  \begin{align}\label{holder}
    \|w_\lambda(x)\|_{L^q(\Sigma_\lambda^-)}&\leq C\left(\int_{\Sigma_\lambda^-}\left|f^\prime(u_\lambda)w_\lambda(x)\right|^{\frac{nq}{n+2kq}}\right)^{q+\frac{2k}{n}}\nonumber\\
    &\leq C\|f^\prime(u_\lambda)\|_{L^{\frac{n}{2k}}(\Sigma_\lambda^-)}\|w_\lambda(x)\|_{L^q(\Sigma_\lambda^-)}.
  \end{align}
  Besides, using growth condition on $f$, we get
  \begin{align}\label{reflect_f'}
    &\|f^\prime(u_\lambda)\|_{L^{\frac{n}{2k}}(\Sigma_\lambda^-)}=C\left(\int_{\Sigma_\lambda^-}\left|f^\prime(|J|^{\frac{n-2k}{2n}}u(x^\lambda))\right|^{\frac{n}{2k}}dV_x\right)^{\frac{2k}{n}}\nonumber\\
    &=C\left(\int_{\phi(\Sigma^-_\lambda)}\left|f^\prime(|J(x)|^{\frac{2k-n}{2n}}u(x))\right|^{\frac{n}{2k}}|J(x)|dV_x\right)^{\frac{2k}{n}}\nonumber\\
    &\leq C\left(\int_{B_\lambda\setminus B_{r_0}}\left|f^\prime(u(x))\right|^{\frac{n}{2k}}dV_x\right)^{\frac{2k}{n}}.
  \end{align}
  Since $f^\prime(u)\in L_{loc}^{\frac{n}{2k}}(\hn)$, we can choose $\lambda$ sufficiently small so that
  $C\|f^\prime(u_\lambda)\|_{L^{\frac{n}{2k}}(\Sigma_\lambda^-)}<1$ in \eqref{w_est}.
  Consequently, $\|w_\lambda(x)\|_{\Sigma_\lambda^-}=0$ and $\Sigma_\lambda^-$ has zero measure for $\lambda$ sufficiently small.
\end{proof}

Recall that
\[\lambda^\sharp = 2\arctanh(\tanh^2\frac{\lambda}{2})\]
and define
\[\tilde{B}_{\lambda}(y):=B_\lambda(y)\setminus B_{\lambda^\sharp}(y).\]

Lemma \ref{start1} provides a starting point to move the sphere. We then start to continuously increase the radius $\lambda$
of the geodesic sphere $S_\lambda(P)=\partial B_\lambda(P)$ as long as the inequality
\[w_\lambda(x)\leq 0, \;\forall x\in \tilde{B}_{\lambda}(P)\]
holds. Now given a center $P\in\hn$, define its critical scale $\lambda_0$ by
\begin{equation}
  \lambda_0=\sup\{\mu>0: w_\mu(x)\leq 0, x\in \tilde{B}_{\mu}(P), \mu\leq \lambda\}.
\end{equation}
Clearly, such critical value $\lambda_0$ exists and is either a finite positive number or $+\infty$.

\begin{lemma}\label{critical_case}
  If $\lambda_0<\infty$ for some $P\in\hn$, then $u_{\lambda_0,P}(x)=u(x)$ in $\hn\setminus B_{\lambda_0^\sharp}(P)$.
\end{lemma}

\begin{proof}
  Again for simplicity, we drop the subscript $P$ whenever it causes no confusion. 
  In fact, it suffices to show
  \begin{equation}\label{eq4.2}
  u_{\lambda_0}(x)=u(x),\; \forall x\in B^c_{\lambda_0}.
  \end{equation}
 Suppose this is not true, from \eqref{holder} and \eqref{reflect_f'}, we get
 \begin{align}
  \|u_\lambda(x)-u(x)\|_{L^q(\Sigma_\lambda^-)}&\leq C\left(\int_{\phi(\Sigma^-_\lambda)}\left|f^\prime(u(x))\right|^{\frac{n}{2k}}dV_x\right)^{\frac{2k}{n}}\|u_\lambda(x)-u(x)\|_{L^q(\Sigma_\lambda^-)}.
  \end{align}
  Denote by $\chi_E$ the characteristic function of the set $E$, then $\chi_{\Sigma^-_\lambda} = 0$ almost everywhere whenever $\lambda < \lambda_0$.
  By Lebesgue's dominated theorem,
  \begin{equation*}
  \int_{\phi(\Sigma^-_\lambda)} |f^\prime(u)|^{\frac{n}{2k}}dV_y = \int_{B_{\lambda+1}} \chi_{\phi(\Sigma^-_\lambda)} |f^\prime(u)|^{\frac{n}{2k}}dV_y \to 0, \text{ as }\lambda\to\lambda_0.
  \end{equation*}
  Thus, there exists $\varepsilon>0$ such that for $\lambda\in [\lambda_0, \lambda_0+\varepsilon)$,
  \begin{equation*}
    C\left(\int_{\phi(\Sigma^-_\lambda)}\left|f^\prime(u(x))\right|^{\frac{n}{2k}}dV_x\right)^{\frac{2k}{n}}<1.
  \end{equation*}
  As a result, $\|u_\lambda(x)-u(x)\|_{L^q(\Sigma_\lambda^-)}=0$ so that $\Sigma_\lambda^-$ has measure zero. In other words,
  \begin{equation*}
    u_\lambda(x)\leq u(x), x\in B^c_\lambda,
  \end{equation*}
  where $\lambda\in[\lambda_0,\lambda_0+\varepsilon)$, which contradicts the definition of $\lambda_0$.
\end{proof}

\begin{lemma}\label{globalness_lambda}
If there exists $P_0$ such that the corresponding critical value $\mu_0<\infty$, then for every $P\in\hn$, the corresponding critical value $\lambda_0<\infty$.
\end{lemma}

\begin{proof}
We argue by contradiction and suppose that there exists $P_1\in\hn$ at which the critical value $\mu_1=+\infty$. Precisely,
choose a sequence $\{\mu_j\}_{j\in\mathbb{Z}_+}$ so that $\lim\limits_{j\to+\infty}\mu_j=+\infty$ and
\begin{equation}\label{c2}
  u_{\mu_j}(x)\geq u(x), \; \forall x\in \tilde{B}_{\mu_j}(P_1).
\end{equation}
On one hand, by Lemma \ref{critical_case}, $u_{\mu_0}(x)=u(x)$, i.e.,
\[|J_{\phi_{P_0,\mu_0}}|^{\frac{n-2k}{2n}}u(x^{\mu_0})=u(x),\; \forall x\in \tilde{B}_{\mu_0}(P_0).\]
Let $R=\rho(x, P_0)$, we see that
\begin{align*}
&\liminf\limits_{R\to\infty}(\sinh R)^{n-2k}u(x) = \lim\limits_{R\to\infty}\left(\frac{\tanh^2\frac{\lambda}{2}}{1-\tanh^4\frac{\lambda}{2}/\tanh^2\frac{R}{2}}\right)^{\frac{n-2k}{2}}u(x^{\mu_0})\\
&=\left(\frac{\tanh^2\frac{\lambda}{2}}{1-\tanh^4\frac{\lambda}{2}}\right)^{\frac{n-2k}{2}}\liminf\limits_{R\to\infty}u(x^{\mu_0})=C>0,
\end{align*}
where $x^{\mu_0}\in\partial B_{\lambda^\sharp}(P)$.
On the other hand, \eqref{c2} implies that
\begin{equation}\label{compare}
  |J_{\phi_{P_0,\mu_j}}|^{\frac{n-2k}{2n}}u(x_1^{\mu_j})\leq u(x_1),
\end{equation}
where $x_1$ is chosen so that $x_1$ is contained in $\tilde{B}_{\mu_j}(P_1)$ for every sufficiently large $j$.
Since
\begin{equation*}
  \lim\limits_{j\to\infty}|J_{\phi_{P_0,\mu_j}}|=\lim\limits_{\mu_j\to\infty}\frac{(\tanh\frac{\mu_j}{2})^2(\sinh \frac{r}{2})^{-2}}{1-\tanh(\frac{\mu_j}{2})^4(\tanh\frac{r}{2})^{-2}}=1,
\end{equation*}
letting $j\to\infty$ in \eqref{compare} yields
\begin{align*}
  u(x_1)&\leq\lim\limits_{j\to\infty}|J_{\phi_{P_0,\mu_j}}|u(x_1^{\mu_j})\\
  &=\lim\limits_{j\to\infty}\frac{1}{(\sinh\mu_j)^{n-2k}}\left(\frac{\tanh^2\frac{\mu_j}{2}}{1-\tanh^4\frac{\mu_j}{2}}\right)^{\frac{n-2k}{2}}u(x^{\mu_j})\\
  &=\lim\limits_{j\to\infty}\frac{1}{(\sinh\mu_j)^{n-2k}}\cdot C=0.
\end{align*}
However, this contradicts $u>0$ everywhere in $\hn$.
\end{proof}

\begin{lemma}\label{constant_case}
  If $\lambda_0=\infty$ for any $P\in\hn$, $u$ is constant.
\end{lemma}

\begin{proof}
  It suffices to show that for any two points $x, y$, we have $u(x)=u(y)$. First notice that there exists a unique geodesic curve $\gamma$ connecting $x$ and $y$. Then for every $z \in \gamma$, there exists a unique $\lambda$ such that the Kelvin transform with respect to $\partial B_{\lambda}(z)$ sends $x$ to $y$. Without loss of generality, we can assume $x \in B_{\lambda}(z)$ and hence we have $u(x) \leq |J_{\phi_{z, \lambda}}|^{\frac{n-2k}{2n}} u(y)$. Notice that both $\rho(x,z)$ and $\rho(y,z)$ are completely determined by $\lambda$ and it is easy to verify that $\rho(x,z), \rho(y,z) \to \infty$ as $\lambda \to \infty$. This happens when $z$ is chosen to be far away from $x,y$. Now the Jacobian
  \begin{equation*}
    \lim\limits_{\lambda\to\infty}|J_{\phi_{z,\lambda}}|=\lim\limits_{\lambda\to\infty}\left[\frac{(\tanh\frac{\lambda}{2})^2(\sinh \frac{\rho(x,z)}{2})^{-2}}{1-\tanh(\frac{\lambda}{2})^4(\tanh\frac{\rho(x,z)}{2})^{-2}}\right]^n=1,
  \end{equation*}
  which implies $u(x)\leq u(y)$. Then the conclusion follows since $x, y$ are arbitrarily chosen.
\end{proof}

\begin{lemma}\label{symm_case}
   If there exists a point $z\in\hn$ such that the corresponding critical value $\mu_0<\infty$, $u$ is radial symmetric with respect to $z$ and is non-increasing.
\end{lemma}

\begin{proof}

From Lemma \ref{critical_case} and \ref{globalness_lambda}, we know that for any picked center $z\in\hn$, there exists a critical scale $\lambda_z\in(0,+\infty)$ depending on $z$ such that
\begin{equation}\label{caseB}
  u(x)= |J_{\phi_{z,\lambda_z}}|^{\frac{n-2k}{2n}}u(x^{\lambda_z}), \quad \forall x\in\hn\setminus B_{\lambda^\sharp_z}(z)
\end{equation}
Recall that in terms of the polar coordinates centered at $z$, $x^{\lambda_z}=(\varphi(r),\theta_z)$ for $x=(r,\theta_z)\in\hn\setminus B_{\lambda_z^\sharp}.$
Denote $x_z$ the intersecting point of the geodesic $\overrightarrow{zx}$ and the sphere $S_{\lambda^\sharp_z}(z)$. We see that there exists a constant $\alpha$ which is defined as
\begin{align}\label{limit1}
  \alpha:&=\lim\limits_{r\to\infty}|\sinh\frac{r}{2}|^{n-2k}u(x) = \lim\limits_{r\to\infty}\left(\frac{\tanh^2\frac{\lambda_z}{2}}{1-\tanh^4\frac{\lambda_z}{2}/\tanh^2\frac{r}{2}}\right)^{\frac{n-2k}{2}}u(x^{\lambda_z})\nonumber\\
  &=\left(\frac{\tanh^2\frac{\lambda_z}{2}}{1-\tanh^4\frac{\lambda_z}{2}}\right)^{\frac{n-2k}{2}}u(x_z).
\end{align}
If $\alpha=0$, then we are done since $u(x)\equiv 0$, which contradicts the positivity of $u$.
Otherwise, we may assume without loss of generality that $\alpha=1$ as we can replace $u$ by a nonzero multiple of $u$.

Let $z$ be the origin $0$, and denote $\Lambda_{\lambda_0}:=\left(\frac{\tanh^2\frac{\lambda_0}{2}}{1-\tanh^4\frac{\lambda_0}{2}}\right)^{\frac{n-2k}{2}}$, $\Lambda_{\lambda_z}:=\left(\frac{\tanh^2\frac{\lambda_z}{2}}{1-\tanh^4\frac{\lambda_z}{2}}\right)^{\frac{n-2k}{2}}$, and $x_0:=x_z\mid_{z=0}$ we can repeat the arguement of \eqref{limit1} and get
\begin{align}\label{xz-relation}
 \Lambda_{\lambda_0}u(x_0) = \lim\limits_{r\to\infty}|\sinh\frac{r}{2}|^{n-2k}u(x) = \Lambda_{\lambda_z}u(x_z)=1.
\end{align}
It is helpful to have the following expression of the Jacobian, which is equivalent to \eqref{Jacobian} by direct computation, 
\begin{equation}
  |J_\lambda|=\Big(\frac{\tanh\frac{\lambda}{2}\cosh\frac{\phi_\lambda(r)}{2}}{\sinh\frac{r}{2}}\Big)^{2n}.
\end{equation}
Now for $r_z:=\rho(x,z)$ sufficiently large and $\theta_z$ fixed, we consider the Taylor expansion of the right hand side of \eqref{caseB} at $x_z$ with respect to the radial variable:
\begin{align}\label{exp1}
  u(x)=&|J_{\phi_{z,\lambda_z}}|^{\frac{n-2k}{2n}}u(x^{\lambda_z})=|J_{\phi_{z,\lambda_z}}|^{\frac{n-2k}{2n}}\left[u(x_z)+\partial_{r_z}u(x_z)\rho(x^{\lambda_z},x_z)+O(\rho^2(x^{\lambda_z},x_z))\right]\nonumber\\
  =&\frac{1}{\left(\sinh\frac{r_z}{2}\right)^{n-2k}}+\frac{1}{(\sinh\frac{r_z}{2})^{n-2k}}\left(\frac{\tanh^2\frac{\lambda_z}{2}}{1-\tanh^4\frac{\lambda_z}{2}/\tanh^2\frac{r_z}{2}}\right)^{\frac{n-2k}{2}}\partial_{r_z}u(x_z)\rho(x^{\lambda_z},x_z)\nonumber\\
  &+|J_{\phi_{z,\lambda_z}}|^{\frac{n-2k}{2n}}O(\rho^2(x^{\lambda_z},x_z)).
\end{align}
Denoting $r_0=\rho(x,0)$, we also obtain the following expansion in the same manner:
\begin{align}\label{exp2}
  u(x)&=|J_{\phi_{0,\lambda_0}}|^{\frac{n-2k}{2n}}\left[u(0)+\partial_{r_0}u(0)\rho(x^{\lambda_0},x_0)+O(\rho(x^{\lambda_0},x_0))\right]\nonumber\\
  =&\frac{1}{(\sinh\frac{r_0}{2})^{n-2k}}+\frac{1}{(\sinh\frac{r_0}{2})^{n-2k}}\left(\frac{\tanh^2\frac{\lambda_0}{2}}{1-\tanh^4\frac{\lambda_0}{2}/\tanh^2\frac{r_0}{2}}\right)^{\frac{n-2k}{2}}\partial_{r_0}u(0)\rho(x^{\lambda_0},x_0)\nonumber\\
  &+|J_{\phi_0,\lambda_0}|^{\frac{n-2k}{2n}}O(\rho^2(x^{\lambda_0},x_0)).
\end{align}

We further expand the first term on the right hand side of \eqref{exp1} at the origin $z=0$:
\begin{align}\label{exp3}
  \frac{1}{(\sinh\frac{r_z}{2})^{n-2k}}&=\frac{1}{(\sinh\frac{r_0}{2})^{n-2k}}-\left(\frac{n-2k}{2}\right)\frac{\cosh(r_0/2)}{(\sinh r_0/2)^{n-2k+1}}g_{\hn}(\nabla \rho(x,z)_{\vert z=0}, |r_z-r_0|)\nonumber\\
  &+O\left\{\left(\frac{\cosh(r_0/2)}{(\sinh r_0/2)^{n-2k+1}}\right)^\prime|r_z-r_0|^2\right\}.
\end{align}

Now combining \eqref{exp1}, \eqref{exp2} and \eqref{exp3}, we obtain that as $x\to\infty$,
\begin{align}\label{ode1}
  \Lambda_{\lambda_z}\partial_{r_z}u(x_z)\rho(x^{\lambda_z},x_z)-\frac{n-2k}{2}\coth\frac{r_0(x)}{2}\rho(z,0)[\partial_{r_z} \rho(z,x)]_{\vert z=0}=\Lambda_{\lambda_0}\partial_{r_0}u(0)\rho(x^{\lambda_0},x_0).
\end{align}
Next we compute 
\begin{align*}
  &\rho(x^{\lambda_z},x_z)=\phi_{\lambda_z}(r)-\lim\limits_{r\to\infty}\phi_{\lambda_z}(r)\\
  &=\int_\infty^r\phi^\prime_{\lambda_z}(t)dt\\
  &=\int_r^\infty\frac{(\tanh\frac{\lambda}{2})^2(\csch \frac{t}{2})^{2}}{1-(\tanh\frac{\lambda}{2})^4(\coth\frac{t}{2})^{2}}dt.
\end{align*}
We then notice that as $x\to\infty$, 
\begin{align*}
  \rho(x^{\lambda_z},x_z)\to\Lambda_{\lambda_z}^{\frac{2}{n-2k}}\int_r^\infty\frac{1}{\sinh^2\frac{t}{2}}dt
  =\Lambda_{\lambda_z}^{\frac{2}{n-2k}}(\coth r-1),
\end{align*}
where $r=|x|$.
Plugging this into \eqref{ode1} gives
\begin{align*}
  u(x_z)^{-\frac{n-2k+2}{n-2k}}\partial_{r_z}u(x_z)&(\coth r-1)-\frac{n-2k}{2}\coth\frac{r_0(x)}{2}\partial_{r_z}r_z(x)_{\vert z=0}|r_z(x)-r_0(x)|\\
  &= u(x_0)^{-\frac{n-2k+2}{n-2k}}\partial_{r_0}u(x_0)(\coth r-1).
\end{align*}
Since $|x|\to\infty$, $\partial_{r_z}$ is essentially the same as $\partial_{r_0}$. Moreover, since the radial direction can be chosen arbitrarily, the above equation reduces to an ordinary differential equation. By the uniqueness of ODE, we see that the solution must have the form given in \cite{Liu}:
\begin{equation*}
  u(x_z)=\frac{\alpha}{(\cosh^2\frac{r(x_z)}{2}+\beta)^{\frac{n-2k}{2}}}.
\end{equation*}

Now as we exhaust $z$ over $\hn$, it can be shown that $u$ is globally defined, namely, for any $x\in\hn$, we have
\begin{equation*}
    u(x)=\frac{\alpha}{(\cosh^2\frac{r(x)}{2}+\beta)^{\frac{n-2k}{2}}},
\end{equation*}
where $r(x)$ is the geodesic distance from $x$ to $P$.
\end{proof}

Combining Lemmas \ref{constant_case} and \ref{symm_case}, the proof of Theorem \ref{thm1} is now concluded.

\subsection{Proof of Theorem \ref{thm2}}

Since no decay condition on $u$ near infinity is assumed, we are unable to perform the method of moving plane directly.
Instead, we consider the method applied to the Kelvin transform of $u$. Let
\begin{equation}
    \bar{u}(x) := |J_{\phi_r}|^{\frac{n-2k}{2n}}u(\tx)
\end{equation}
be the Kelvin transform of $u$ with respect to the geodesic sphere $\{x: \rho(x,P)=1\}$, where $\tilde{x}$ denotes the inversion of $x$.
It is known that GJMS operators enjoy the conformal invariance in the sense of \eqref{pk_conformal}.
Let $r_x=\rho(x,P)$ denote the distance from $x$ to the center $P$ of the geodesic ball.
Now if $u$ satisfies \eqref{p_k}, we get
\begin{equation}\label{pk2}
    P_k \bar{u} = |J_{\phi_{r_x}}|^{\frac{n+2k}{2n}}f(u(\tilde{x})) = |J_{\phi_{r_x}}|^{\frac{n+2k}{2n}}f(|J_{\phi_{r_x}}|^{-\frac{n-2k}{2n}}\bar{u}).
\end{equation}
We use Poincar\'e ball as the model for hyperbolic space $\hn$, and recall the foilation structure on it with the notations:
\[U=\hn\cap\{x_1=0\}, \; U_s=A_s(U),\; \Sigma_\lambda=\bigcup_{s<\lambda}U_s.\]
The reflection with respect to $\Sigma_\lambda$ is denoted as
\[x^\lambda=I_\lambda(x),\quad \lambda\in\mathbb{R}.\]
We further define that for $\lambda<x_1(P)$,
\begin{equation}\label{nota}
  u^\lambda(x) = u(x^\lambda), \;\bar{u}^\lambda(x)=\bar{u}\circ I_\lambda(x), \text{ and } v^\lambda(x) = \bar{u}^\lambda(x)-\bar{u}(x).
\end{equation}
By the definition of Kelvin transform and $v^\lambda$, we immediately notice that
\begin{equation*}
    \lim\limits_{|x|\to\infty}v^\lambda(x)=0.
\end{equation*}

\begin{remark}
  In the model case where $f(u)=u^p$, \eqref{pk2} becomes
  \begin{equation*}
      P_k \bar{u} = |J_{\phi_{r_x}}|^{\tau}f(\bar{u}),
  \end{equation*}
  with
  \begin{equation*}
      c(x) = -p|J_{\phi_{r_x}}|^{\tau}\bar{u}^{p-1}(x),
  \end{equation*}
  where $$\tau=\frac{n+2k}{2n}-p\cdot\frac{n-2k}{2n}.$$
  Clearly, $\tau=0$ in the critical case and $\tau>0$ in the subcritical case.
\end{remark}

\textbf{The subcritical case.}

For $1<p<\frac{n+2k}{n-2k}$, we show that \eqref{p_k} admits no positive solution.
We first show that the moving plane procedure can be initiated from the infinity.
\begin{lemma}\label{init2}
  Let $u$ be a nonnegative solution of \eqref{p_k} in the subcritical case, then
  for $\lambda<0$ sufficiently large,
  \begin{equation}\label{s1}
    v^\lambda(x)\leq 0, \; x\in \Sigma_\lambda\setminus\{P^\lambda\}.
\end{equation}
\end{lemma}

\begin{proof}
Setting
\[\Sigma_\lambda^-=\{x\in\Sigma_\lambda: v^\lambda(x)>0\},\]
then it suffices to show that $\Sigma_\lambda^-$ has zero measure.

Recall that $G(x,y)$ denotes the Green's function of $P_k$, and combine \eqref{pk2} and \eqref{nota}, we see that
\begin{equation}
  \bar{u}^\lambda(x)=\int_{\bn}G(x,y)|J_{x^\lambda}|^{\frac{n+2k}{2n}}f(u^\lambda(y))dV_y,
\end{equation}
where we put $J_{\phi_{r_x}}=J_x$ for simplicity.
Thus,
\begin{align*}
    v^\lambda(x)&=\int_{\Sigma_\lambda} G(x,y)|J_{y^\lambda}|^{\frac{n+2k}{2n}}f(u^\lambda(\tilde{y}))dV_y+\int_{\Sigma^c_\lambda} G(x,y)|J_{y^\lambda}|^{\frac{n+2k}{2n}}f(u^\lambda(\tilde{y}))dV_y\\
    &\quad -\int_{\Sigma_\lambda} G(x,y)|J_{y}|^{\frac{n+2k}{2n}}f(u(\tilde{y}))dV_y-\int_{\Sigma^c_\lambda} G(x,y)|J_{y}|^{\frac{n+2k}{2n}}f(u(\tilde{y}))dV_y\\
    &=\int_{\Sigma_\lambda} G(x,y)|J_{y^\lambda}|^{\frac{n+2k}{2n}}f(u^\lambda(\tilde{y}))dV_y+\int_{\Sigma_\lambda} G(x,y^\lambda)|J_{y}|^{\frac{n+2k}{2n}}f(u(\tilde{y}))dV_y\\
    &\quad -\int_{\Sigma_\lambda} G(x,y)|J_{y}|^{\frac{n+2k}{2n}}f(u(\tilde{y}))dV_y-\int_{\Sigma_\lambda} G(x,y^\lambda)|J_{y^\lambda}|^{\frac{n+2k}{2n}}f(u(\tilde{y}^\lambda))dV_y.
\end{align*}
Noting $G(x,y^\lambda)=G(x^\lambda,y)$, since $I_t$ is an isometry so that $\rho(x,y^\lambda)=\rho(x^\lambda,y)$.
we get
\begin{align}
  v^\lambda(x) &=\int_{\Sigma_\lambda} \left(G(x,y)-G(x^\lambda,y)\right)\left(|J_{y^\lambda}|^{\frac{n+2k}{2n}}f(u^\lambda(\tilde{y}))-|J_{y}|^{\frac{n+2k}{2n}}f(u(\tilde{y}))\right)dV_y.
\end{align}
We first observe that
\[1>|J_{y^\lambda}|>|J_y|\to 0, \text{ for } y\in\Sigma_\lambda, \text{ as } \lambda\to -\infty.\]
Now in the subcritical case, we have that $\lim\limits_{t\to\infty}f(t)/t^p=c<\infty$ for $1<p<\frac{n+2k}{n-2k}$.
Moreover, $f$ is nondecreasing, and $\bar{u}^\lambda>\bar{u}$ in $\Sigma^-_\lambda$ so that $f(\bar{u}^\lambda(y))-f(\bar{u}(y))>0$.
We next notice that
$$G(x,y)-G(x^\lambda,y)>0,$$
since $\rho(x,y) < \rho(x^\lambda,y)$ for $x, y\in\Sigma_\lambda$,
Now if $v^\lambda\geq 0$, we obtain
\begin{align}\label{vlambda}
  v^\lambda(x)&\lesssim\int_{\Sigma_\lambda} \left(G(x,y)-G(x^\lambda,y)\right)\left(f(|J_{y_\lambda}|^{\frac{n-2k}{2n}}u^\lambda(\tilde{y}))-f(|J_{y}|^{\frac{n-2k}{2n}}u(\tilde{y}))\right)dV_y\nonumber\\
  &=\int_{\Sigma_\lambda} \left(G(x,y)-G(x^\lambda,y)\right)\left(f(\bar{u}^\lambda(y))-f(\bar{u}(y))\right)dV_y.
\end{align}
\eqref{vlambda} still holds if $v^\lambda<0$ since its righ-hand side is nonnegative.

Moreover,
\begin{align*}
  v^{\lambda}(x)
  &\leq \int_{\Sigma^-_\lambda} G(x,y)\left(f(\bar{u}^\lambda(y))-f(\bar{u}(y))\right)dV_y.
\end{align*}
Taking advantage of Green's function's estimate \eqref{pk_est},
\begin{equation*}
  v^{\lambda}(x)\lesssim \int_{\Sigma^-_\lambda}\left(\frac{1}{2\sinh\frac{\rho(x,y)}{2}}\right)^{n-2k}\left(f(\bar{u}^\lambda(y))-f(\bar{u}(y))\right)dV_y.
\end{equation*}
Now due to Hardy-Littlewood-Sobolev inequality on $\hn$, for some $q>\frac{n}{n-2k}$,
  \begin{align*}
    \|v^\lambda(x)\|_{L^q(\Sigma_\lambda^-)}\lesssim\|f(\bar{u}^\lambda)-f(\bar{u})\|_{L^{\frac{nq}{n+2kq}}(\Sigma_\lambda^-)}.
  \end{align*}
Applying mean value theorem and H\"older's inequality, we have
\begin{align}\label{vol_est}
  \|v^\lambda(x)\|_{L^q(\Sigma_\lambda^-)}&=\norm{v^\lambda(x)f^\prime(\xi)}_{L^{\frac{nq}{n+2kq}}(\Sigma_\lambda^-)}\nonumber\\
  &\lesssim\|f^\prime(\xi)\|_{L^\frac{n}{2k}(\Sigma_\lambda^-)}\|v^\lambda(x)\|_{L^q(\Sigma_\lambda^-)},
\end{align}
where $\xi=\theta\bar{u}+(1-\theta)\bar{u}_\lambda$ with $\theta\in(0,1)$.
Since $u$ in locally $L^{\frac{2n}{n-2k}}$, then $\bar{u}$ has no singularity at infinity. Now with $f^\prime(t)\in L_{loc}^{\frac{n}{2k}}(\hn)$, one can choose $N$ sufficiently large such that, for $\lambda<-N$,
\[C\norm{f^\prime(\xi(t))}_{L^{\frac{n}{2k}}(\Sigma_\lambda^-)}<\frac{1}{2}.\]

Thus, $\|v^\lambda(x)\|_{L^q(\Sigma_\lambda^-)}=0$ for $\lambda$ sufficiently negative.
 Consequently, $\Sigma_\lambda^-$ has zero measure and the proof is concluded.
\end{proof}

Lemma \ref{init2} provides a starting point, from which we are able to move the plane $\Sigma_\lambda$ to the right
(positive $x_1$-direction) as long as \eqref{s1} holds until a limiting position.
Define
\[\Lambda:=\sup\{\lambda<x_1(P): v^{\mu}(x)\leq 0, \forall x\in\Sigma_\mu\setminus\{P^\lambda\}, \mu\leq\lambda\}.\]
We will show that
\begin{equation}\label{limit2}
  \Lambda=x_1(P)
\end{equation}
and
\begin{equation*}
  v^{\Lambda}(x)\equiv  0,\; x\in\Sigma_\Lambda\setminus\{P^\Lambda\}.
\end{equation*}
Suppose not, if $\Lambda<x_1(P)$
it can be shown that the plane $T_\lambda$ can be moved farther to the right, i.e. there exists $\varepsilon>0$ such that
for any $\lambda\in [\Lambda,\Lambda+\varepsilon)$, there holds
\begin{equation*}
    v^\lambda(x)\leq 0,\; x\in \Sigma_\lambda\setminus\{P^\lambda\}.
\end{equation*}

Similar to \eqref{vol_est},
we have the estimate when $\lambda\in[\Lambda,\Lambda+\varepsilon)$, for some $q>\frac{n}{n-2k}$
\begin{equation*}
  \|v^\lambda(x)\|_{L^q(\Sigma_\lambda^-)}\leq C\|f^\prime(\bar{u}_\lambda)\|_{L^\frac{n}{2k}(\Sigma_\lambda^-)}\|v^\lambda(x)\|_{L^q(\Sigma_\lambda^-)}.
\end{equation*}
Denote $\chi_E$ the characteristic function of the set $E$, then $\chi_{\Sigma^-_\lambda}\to 0$ almost everywhere when $\lambda\to\Lambda$.
Let $(\Sigma^-_\lambda)^*$ be the reflection of $\Sigma^-_\lambda$ about the hyperplane $T_\lambda$.
Again, since $u$ is locally $L^{\frac{2n}{n-2k}}$, for any domain $\Omega$ that has positive distance from $P$, we have
\[\int_{\Omega}|\bar{u}|^{\frac{2n}{n-2k}}dV_y<\infty.\]
Then by Lebesgue's dominated theorem and $f^\prime(t)\in L_{loc}^{\frac{n}{2k}}(\hn)$,
\begin{equation*}
\int_{\Sigma^-_\lambda} |f^\prime(\bar{u}_\lambda)|^{\frac{n}{2k}}dV_x = \int_{\Sigma_\lambda^c} \chi_{(\Sigma^-_\lambda)^*} |f^\prime(\bar{u})|^{\frac{n}{2k}}dV_x \to 0, \text{ as }\lambda\to\lambda_0.
\end{equation*}
Thus, there exists $\varepsilon>0$ such that for $\lambda\in [\lambda_0, \lambda_0+\varepsilon)$,
\begin{equation*}
  C\left(\int_{\Sigma^-_\lambda}\left|f^\prime(\bar{u}_\lambda)\right|^{\frac{n}{2k}}dV_x\right)^{\frac{2k}{n}}<1.
\end{equation*}

Thus, we conclude that $\norm{v^\lambda}_{L^q(\Sigma^-_\lambda)}=0$, where $\lambda>\Lambda$. This means we can keep moving the plane while $v^\lambda\leq 0$ is not violated,
until we reach $x_1(P)$, and obtain
\[v^{x_1(P)}(x)\leq  0,\; x\in\Sigma_{x_1(P)}.\]
Similarly, one can move the plane $T_\lambda$ from $+\infty$ to the left and show that
\[v^{x_1(P)}(x)\geq  0,\; x\in\Sigma_{x_1(P)}.\]
Therefore,
\[v^{x_1(P)}(x)\equiv 0,\; x\in\Sigma_{x_1(P)}.\]
Moreover, since $x_1$ direction can be chosen arbitrarily, we have actually shown that $\bar{u}$ is radially symmetric with respect to $P$.
This implies that $u$ is also radially symmetric about $P$. Since $P$ is arbitrarily chosen, we conclude that $u$ is constant.
Moreover, constant solution which satisfies \eqref{p_k} must be $u\equiv 0$.

\textbf{Critical Case.}

Let $\bar{u}$ be the Kelvin transform of $u$ centered at the origin, then
\begin{equation}
  P_k \bar{u} = |J_{\phi_r}|^{\frac{n+2k}{2n}}f(|J_{\phi_r}|^{-\frac{n-2k}{2n}}\bar{u}).
\end{equation}
The following is immediate since it is entirely the same as that in the subcritical case:
\begin{equation*}
  v^\lambda(x)\leq 0, \forall x\in\Sigma_\lambda\setminus\{0^\lambda\} \text{ for } \lambda<0 \text{ sufficiently large}.
\end{equation*}
Let
\[\Lambda:=\sup\{\lambda\leq 0: v^{\mu}(x)\leq 0, \forall x\in\Sigma_\mu\setminus\{P^\lambda\}, \mu\leq\lambda\}.\]
If $\Lambda<0$, similar to the subcritical case, we can show that $v^{\Lambda}\equiv 0, \forall x\in\Sigma_{\Lambda}$.
This means $0$ is not a singular point of $\bar{u}$ so that $u(x)=O(e^{-\frac{n-2k}{2n}r})$ when $|x|\to\infty$.
Then it follows the standard moving plane arguement \cite{LLW} that $u$ is symmetric about some point in $\hn$.

If $\Lambda=0$, then by moving the plane from $x_1=+\infty$ we are able to show that $\bar{u}$ is symmetric about the origin and so is $u$.
In any case, $u$ is symmetric about some point and is decreasing in the radial variable.

\subsection{Proof of Theorem \ref{thm3}}

 The pullback function of \eqref{soln} in $\rn$ is (up to rescaling)
\begin{equation}
  \phi(x)=\left(\frac{2a}{a^2+|x|^2}\right)^{\frac{n}{2}-k},
\end{equation}
by direct computation, or see \cite{Liu}. According to Theorem \ref{thm1}, the solution of \eqref{p_k} on $\hn$ is either \eqref{soln} or constant. If \eqref{radial_eq_eucl} admits a positive radial solution $u$, then $u=\phi$, and $\beta_1=\beta$ by definition. Since $U(0)=\alpha=(2/a)^{\frac{n}{2}-k}$

Otherwise, the solution of \eqref{p_k} degenerates to a constant function whose pullback in $\rn$ blows up on the boundary with $\beta_1<\beta$. For $\beta>\beta_1$, no such solution exists due to the explicit form of $u$.

\subsection{Proof of Corollary \ref{cor2}}
Similar to the proof of Theorem \ref{thm1}, we apply the method of moving spheres.  
Denote \[w_\lambda(x)= u_{\lambda}(x) - u(x)\]
and $q=\frac{n+2k}{n-2k}$, we have 
\begin{align}\label{inteq00}
  w_\lambda(x)&=\int_{ \hn\setminus B_{r_0}} G(x,y)\left(|J_{\phi_{\lambda}}|^{\frac{n+2k}{2n}}\tilde{Q}(x^\lambda)u^{q}(x^\lambda) - \tilde{Q}u^{q}\right)dV_y-\int_{B_{r_0}}G(x,y)\tilde{Q}u^qdV_y.
\end{align}
\begin{lemma}\
  For sufficiently small $\lambda>0$, there holds
  \begin{equation*}
    w_\lambda(x)\leq 0,\; \forall x\in B^c_\lambda(P).
  \end{equation*}
\end{lemma}

\begin{proof}
  Let $\lambda>0$ and define
  \begin{equation*}
    \Sigma_\lambda^-=\{x\in B^c_\lambda(P):w_\lambda(x)>0\}.
  \end{equation*}
  In order to prove the lemma, it suffices to show that $\Sigma_\lambda^-$ has zero measure for sufficiently small $\lambda$.

  Noting the decomposition
  \[\hn\setminus B_{r_0}(P)=\phi(B^c_\lambda(P))\cup B^c_\lambda(P),\]
  we rewrite \eqref{inteq00} as
  \begin{equation}\label{inteq11}
    w_\lambda(x)=\int_{\phi(B^c_\lambda(P))\cup B^c_\lambda(P)} G(x,y)\left(|J_{\phi_{\lambda}}|^{\frac{n+2k}{2n}}\tilde{Q}(x^\lambda)u^{q}(x^\lambda) - \tilde{Q}u^q\right)dV_y-\int_{B_{r_0}}G(x,y)\tilde{Q}u^qdV_y.
  \end{equation}
  \medskip
  Next, the following is obtained by change of variables $y=z^\lambda, dV_y=|J_{\phi_\lambda}(z)|dV_z$.
  \begin{align}\label{inteq22}
    &\int_{\phi(B^c_\lambda(P))} G(x,y)\left( |J_{\phi_{\lambda}}(y)|^{\frac{n+2k}{2n}}\tilde{Q}(y^\lambda)u^q(y^\lambda) - \tilde{Q}u^q\right)dV_y \nonumber \\
    =& \int_{B^c_\lambda(P)} G(x,z^\lambda) \left(|J_{\phi_\lambda} (z^\lambda)|^{\frac{n+2k}{2n}}\tilde{Q}u^q - \tilde{Q}(z^\lambda)u^q(z^\lambda) \right) |J_{\phi_\lambda}(z)| dV_z\nonumber\\
    =&\int_{B^c_\lambda(P)} -|J_{\phi_\lambda}(z)|^{\frac{n-2k}{2n}} G(x,z^\lambda) \left(|J_{\phi_\lambda}(z)|^{\frac{n+2k}{2n}}\tilde{Q}(z^\lambda)u^{q}(z^\lambda) - \tilde{Q}(z)u^q(z) \right) dV_z.
  \end{align}

  Meanwhile, by noticing $\rho(x,y)>\lambda-r_0$ for $x\in\Sigma_\lambda^-$ and $y\in B_{r_0}(P)$,
  Green's function estimate of $P_k$ implies that
  \begin{equation*}
    G(x,y)\lesssim \frac{1}{\sinh(\lambda-r_0)^{n-2k}}\sim\frac{1}{(\lambda-r_0)^{n-2k}} \text{ as } r_0<\lambda\to 0.
  \end{equation*}
  Together with Lebesgue differentiation theorem, we obtain that
  \begin{align}\label{Control_on_B_r0}
    &\int_{B_{r_0}(P)} G(x,y)\tilde{Q}u^q(y)dV_y \lesssim \frac{1}{(\lambda-r_0)^{n-2k}}\int_{B_{r_0}(P)}\tilde{Q}u^q(y)dV_y\nonumber\\
    &\longrightarrow \frac{vol(B_{r_0})}{(\lambda-r_0)^{n-2k}}\tilde{Q}(P)u^q(P), \text{ as } r_0\to 0.
  \end{align}
As shown in the previous proof, this integral tends to $0$ for $r_0$ small.

Combining \eqref{inteq11}, \eqref{inteq22} and \eqref{Control_on_B_r0}, we get
\begin{align*}
  w_\lambda(x)=u_\lambda(x)-u(x)&= \int_{B^c_\lambda(P)}K(x,y)\left(|J_{\phi_\lambda} (y)|^{\frac{n+2k}{2n}}\tilde{Q}(y^\lambda)u^{q}(y^\lambda) - \tilde{Q}(y)u^q(y)\right)dV_y.
\end{align*}
where
\begin{equation*}
  K(x,y) = G(x,y) - |J_{\phi_\lambda}(y)|^{\frac{n-2k}{2n}}G(x,y^\lambda)>0,\quad x, y\in B^c_\lambda(P).
\end{equation*}

We then obtain the following when $x\in\Sigma^-_\lambda$,
  \begin{align*}
    0<w_\lambda(x) &= \int_{B^c_\lambda(P)}K(x,y)\left(|J_{\phi_\lambda} (y)|^{\frac{n+2k}{2n}}\tilde{Q}(y^\lambda)u^q(y^\lambda) - \tilde{Q}u^q\right)dV_y\nonumber\\
    &= \int_{B^c_\lambda(P)}K(x,y)\left(\tilde{Q}(y^\lambda)u^q_\lambda(y) - \tilde{Q}(y)u^q(y)\right)dV_y\nonumber\\
    &\leq \int_{B^c_\lambda(P)}K(x,y)\left(u^q_\lambda(y) - u^q(y)\right)dV_y,
  \end{align*}
where the last inequality follows the monotonicity and integrability of $\tilde{Q}$.
Moreover, by the definition of $K(x,y)$ and Green's function estimate for $P_k$,
  \begin{align*}
    w_\lambda(x)
    &\lesssim \int_{\Sigma_\lambda^-}G(x,y)\left(u^q_\lambda(y) - u^q(y)\right)dV_y\\
     &\leq \int_{\Sigma_\lambda^-}\left(\frac{1}{2\sinh\frac{\rho(x,y)}{2}}\right)^{n-2k}\left(u^q_\lambda(y) - u^q(y)\right)dV_y
  \end{align*}

Following the Hardy-Littlewood-Sobolev inequality on $\hn$ and H\"older's inequality, we get
  \begin{align*}
    \|w_\lambda(x)\|_{L^r(\Sigma_\lambda^-)}&\leq C\|u^q_\lambda(y) - u^q(y)\|_{L^{\frac{nq}{n+2kq}}(\Sigma^-_\lambda)}\\
    &\leq C\| u^{q-1}_\lambda(u_\lambda-u)\|_{L^{\frac{nq}{n+2kq}}(\Sigma^-_\lambda)}\\
    &\leq C\|u^{q-1}_\lambda(x)\|_{L^{\frac{n}{2k}}(\Sigma_\lambda^-)}\|w_\lambda(x)\|_{L^r(\Sigma_\lambda^-)}.
  \end{align*}
  Besides, notice that
  \begin{align*}
    &\|u^{q-1}_\lambda(x)\|_{L^{\frac{n}{2k}}(\Sigma_\lambda^-)}=C\left(\int_{\Sigma_\lambda^-}\left||J|^{\frac{2k}{n}}u^{q-1}(x^\lambda)\right|^{\frac{n}{2k}}dV_x\right)^{\frac{2k}{n}}\\
    &=C\left(\int_{\phi(\Sigma^-_\lambda)}|u^{q-1}(x)|^{\frac{n}{2k}}dV_x\right)^{\frac{2k}{n}}\leq C\left(\int_{B_\lambda\setminus B_{r_0}}|u(x)|^{\frac{2n}{n-2k}}dV_x\right)^{\frac{2k}{n}}.
  \end{align*}
  Since $u\in L_{loc}^{\frac{2n}{n-2k}}(\hn)$, we can choose $\lambda$ sufficiently small so that
  $C\|u\|_{L^{\frac{2n}{n-2k}}(\Sigma_\lambda^-)}<1$.
  Consequently, $\|w_\lambda(x)\|_{\Sigma_\lambda^-}=0$ and $\Sigma_\lambda^-$ has zero measure for $\lambda$ sufficiently small.
\end{proof}

The rest of proof is similar to the proof of Theorem \ref{thm1}. Since $u$ is radial, $\tilde{Q}u^{\frac{n+2k}{n-2k}}$ must take the form $cu^{\frac{n+2k}{n-2k}}$, i.e., $\tilde{Q}$ is constant.

\section*{Acknowledgements}
The authors wish to thank Debdip Ganguly for  the valuable discussions  on the part of the prescribed $Q$-curvature problem on the hyperbolic space.

\end{document}